\documentclass[12pt]{article}

\usepackage[T1]{fontenc}
\usepackage{lmodern}
\usepackage{fancybox}
\usepackage{amsmath,amstext,amssymb,graphicx,graphics,color}
\usepackage{cite}
\usepackage{theorem}
\theorembodyfont{\normalfont\slshape}
\newtheorem{thm}{Theorem}[section]
\newtheorem{proposition}[thm]{Proposition}
\newtheorem{lemma}[thm]{Lemma}
\theoremstyle{break}
\newenvironment{proof}%
{{\par\noindent \it Proof. \nobreak}}%
{\nobreak \removelastskip \nobreak \hfill $\Box$ \medbreak}
\newenvironment{proof_formal}%
{{\par\noindent \it Proof (formal). \nobreak}}%
{\nobreak \removelastskip \nobreak \hfill $\Box$ \medbreak}
\newenvironment{remark}{\par \medskip \noindent {\bf Remark. }\nobreak}{\par \medskip}

\def\paragraph#1{{\bf #1\ }}

\newcommand{\ep}{\varepsilon}
\newcommand{\ld}{\lambda}
\newcommand{\expo}{\mathrm{e}}
\newcommand{\ds}{\displaystyle}
\newcommand{\nb}{\nonumber}

\date{}

\begin{document}

\title{Numerical simulations of a non-conservative hyperbolic system with geometric
  constraints describing swarming behavior}

\renewcommand{\thefootnote}{\fnsymbol{footnote}}

\author{Sebastien Motsch\footnotemark[1] \and Laurent Navoret\footnotemark[2] \footnotemark[3]}
\footnotetext[1]{Center of Scientific Computation and Mathematical Modeling (CSCAMM), University of Maryland,
  College Park, MD 20742, USA, \url{smotsch@cscamm.upd.edu}}
\footnotetext[2]{Universit\'{e} de Toulouse, UPS, INSA, UT1, UTM, Institut de
  Math\'{e}matiques de Toulouse, F-31062 Toulouse, France, \url{laurent.navoret@math.univ-toulouse.fr}}
\footnotetext[3]{CNRS, Institut de Math\'{e}matiques de Toulouse UMR 5219,
  F-31062 Toulouse, France}

\renewcommand{\thefootnote}{\arabic{footnote}}

\maketitle

\begin{abstract}
  The Vicsek model is a very popular individual based model which describes collective
  behavior among animal societies. A macroscopic version of the Vicsek model has been
  derived from a large scale limit of this individual based model \cite{degond2007ml}. In
  this work, we want to numerically validate this Macroscopic Vicsek model (MV). To this
  aim, we compare the simulations of the macroscopic and microscopic models one with each
  other. The MV model is a non-conservative hyperbolic equation with a geometric
  constraint. Due to the lack of theory for this kind of equations, we derive several
  equivalents for this system leading to specific numerical schemes. The numerical
  simulations reveal that the microscopic and macroscopic models are in good agreement
  provided that we choose one of the proposed formulations based on a relaxation of the
  geometric constraint. This confirms the relevance of the macroscopic equation but it
  also calls for a better theoretical understanding of this type of equations.
\end{abstract}



% 35Q80 : Applications of PDE in areas other than physics
% 35L60 : Nonlinear first-order PDE of hyperbolic type
% 35L65 : Conservation laws
% 35L67 : Shocks and singularitiesthe 
% 65M60 : Finite elements, Rayleigh-Ritz and Galerkin methods, finite methods
% 82C22 : Interacting particle systems
% 82C70 : Transport processes
% 82C80 : Numerical methods (Monte Carlo, series resummation, etc.)
% 92D50 : Animal behavior

{\small \noindent {\bf Acknowledgments: } The authors wish to thank Pierre Degond for his support
and his fruitful suggestions. They also would like to thank Guy Theraulaz, Jacques Gautrais and
Richard Bon for helpful discussions. This work was supported by the French 'Agence Nationale pour la Recherche
(ANR)' in the frame of the contract ANR-07-BLAN-0208-03 entitled 'PANURGE'.}

\section{Introduction}

This paper is devoted to the numerical study of a macroscopic version of the Vicsek model
which describes swarming behavior. This macroscopic model has been derived in
\cite{degond2007ml} from the microscopic Vicsek model \cite{vicsek1995ntp}. The goal of
this work is to provide a numerical validation of the macroscopic model by comparing it
with simulations of the microscopic model.

The Vicsek model \cite{vicsek1995ntp} is widely used to describe swarming behavior such as
flock of birds \cite{ballerini2008ira}, schools of fish
\cite{aoki1982sss,reynolds1987fha,huth1992smf,couzin2002cma} (in this case the model is
combined with an attractive-repulsive force) or recently the motion of locusts
\cite{buhl2006dom}. In this model, individuals have a constant velocity and they tend to
align with their neighbors. Despite the simplicity of the model, a lot of questions remain
open about it. A first field of research concerns phase transitions within the model
%% between the state where individual are aligned 
depending on the level of noise
\cite{vicsek1995ntp,gregoire2004oca,nagy2007nac,chate2008cms}.
% the question is to understand whether there is a first or second order phase transition
Another question arises from the long time dynamics of the model
\cite{cucker2007ebf,cucker2007fne,ha2008pka}: is there convergence to a stationary state
of the system? From another perspective, since collective displacements in natural
environment can concern up to several million individuals, it is natural to look for a
macroscopic version of the Vicsek model. On the one hand, macroscopic models constitute
powerful analytical tools to study the dynamics at large scales
\cite{mogilner1999nlm,chuang2007sta,degond2007mls}. On the other hand, the related
numerical schemes are computationally much more efficient compared with particle
simulations of a large number of interacting agents. In \cite{degond2007ml}, a Macroscopic
Vicsek model (MV) has been derived from a large scale limit of the 
microscopic Vicsek model. The macroscopic model is obtained from a rigorous perturbation theory of the
original Vicsek model. Another macroscopic model is obtained in \cite{bertin2006bah} based
on more phenomenological closure assumptions.

The MV model presents several specificities which make the model interesting. First, it
is a non-conservative hyperbolic system and secondly it involves a geometric
constraint. These are the consequences at the macroscopic level of two specificities of the
microscopic model: the total momentum is not conserved by the particle dynamics and the speed of
the particles is constant.  The first property is an intrinsic property of self-propelled
particles and the second property is a usual assumption in the models of collective
displacements \cite{vicsek1995ntp,couzin2002cma,gautrais2009afm}. Up to our knowledge the
theory of such systems is almost empty. Non-conservative systems have been studied in
the literature \cite{bouchut_nonlinear_2004,Lefloch1988ews,chen1994hcl} but none of them
involve geometric constraints.

In this work, since a theoretical framework for such systems is not available, we adopt
several approaches. First, we introduce a conservative formulation for the 1D formulation
of the MV model which is equivalent to the initial one for smooth solutions only. With this conservative
formulation, we can use standard hyperbolic theory to build Riemann problem solution and
shock capturing schemes \cite{leveque1992nmc}. The numerical scheme based on the
conservative formulation is called {\it conservative method}. But since the equivalence
with the original formulation is only valid for smooth solutions \cite{leveque2002fvm},
there is no guarantee that the conservative formulation gives the right answer at
shocks. For this reason, we introduce another formulation of the MV model where the
constraint is treated through the relaxation limit of an unconstrained conservative
system. This formulation leads to a natural numerical scheme based on a splitting between
the conservative part of the equation and the relaxation. This scheme will be referred to
as the {\it splitting method}. For comparison purposes, two
other numerical schemes are also used, an {\it upwind scheme} and a {\it
  semi-conservative} one (where only the mass conservation equation is treated in a
conservative way).

The numerical simulations of the MV model reveal that the numerical schemes all agree on
rarefaction waves but disagree on shock waves. To determine the correct solution, we use
the microscopic model in a regime where its solution is close to that of the macroscopic
model. In practice, this corresponds to regimes where the number of particles per domain
of interaction is high. The splitting method turns out to be in good agreement with
particle simulations of the microscopic model, by contrast with the other schemes. In
particular, for an initial condition with a contact discontinuity, the solution given by
the conservative form is simply a convection of the initial condition whereas the
splitting method and the particle simulations agree on a different and more complex
solution.

These results show first that the MV model well describes the microscopic model in the
dense regime. Secondly, that the correct formulation of the MV model is given by
the limit of a conservative equation with a stiff relaxation term.

The theoretical and numerical studies of the MV model highlight the specificity
of non-conservative hyperbolic models with geometric constraints. More theoretical work is
necessary in order to understand why the splitting method matches the microscopic model whereas the
other methods do not. In particular, an extension of the theory developed in \cite{chen1994hcl} to
non-conservative relaxed models would be highly desirable.

The outline of the paper is as follows: first, we present the Vicsek and MV models in
section \ref{sec:presentation_MV}. Then, we analyze the MV model and give two different
formulations of the model in section \ref{sec:macro_equa}. We develop different numerical
schemes based on these formulations and we use them to numerically solve different Riemann
problems in section \ref{sec:num_simu_cva_macro}.
Finally, we compare simulations of the microscopic model with those of the
macroscopic system in the same situations in section \ref{sec:micro_cva}. Finally, we draw
a conclusion.

\section[{Presentation of the Vicsek and Macroscopic Vicsek models}]{Presentation of the Vicsek and Macroscopic \\ \mbox{Vicsek} models}
\label{sec:presentation_MV}
\setcounter{equation}{0}

At the particle level, the Vicsek model describes the motion of particles which tend to
align with their neighbors. We denote by $x_k$ the position vector
of the $k^{th}$ particle and by $\omega_k$ its velocity with a constant
speed ($|\omega_k|=1$). To simplify, we suppose that the particles move in a plane.
Therefore $x_k\in \mathbb{R}^2$ and $\omega_k\in S^1$. The Vicsek model at the microscopic level is given by the
following equations (in dimensionless variables):
\begin{eqnarray}
  & & \frac{d x_k}{dt} =  \omega_k , \label{Cont_pos_bis} \\
  & & d \omega_k = (\mbox{Id} - \omega_k \otimes \omega_k) 
  (\bar \omega_k\,dt + \sqrt{2d}\, d \! B_t) , 
  \label{Cont_orient_bis}
\end{eqnarray}
where $\mbox{Id}$ is the identity matrix and the symbol $\otimes$ denotes the
tensor product of vectors. $d$ is the intensity of noise, $B_t$ is the Brownian motion
and $\bar \omega_k$ is the direction of mean velocity around the $\text{k}^{\text{th}}$
particle defined by:
\begin{eqnarray}
  & & \bar \omega_k = \frac{J_k}{|J_k|} \;, \qquad J_k = \!\! 
  \sum_{j,  \, |x_j - x_k| \leq R} \!\! \omega_j ,
  \label{CV_moy_3_bis} 
\end{eqnarray}
where $R$ defines the radius of the interaction region. Equation (\ref{Cont_orient_bis})
expresses the tendency of particles to move in the same direction as their neighbors. The
operator $(\mbox{Id} - \omega_k \otimes \omega_k)$ is the orthogonal projector onto the
plane perpendicular to $\omega_k$. It ensures that the speed of particles remains
constant. This model is already a modification of the original Vicsek model
\cite{vicsek1995ntp}, which is a time-discrete algorithm.
% where the interaction period equals the time step.

The Macroscopic Vicsek model (MV) describes the evolution of two macroscopic quantities: the
density of particles $\rho$ and the direction of the flow $\Omega$. The evolution
of $\rho$ and $\Omega$ is governed by the following equations:
\begin{eqnarray}
  &&\partial_t \rho + \nabla_x \cdot (c_1 \rho \Omega)  = 0, \label{eq:macro_rho}\\
  &&\rho \, \left( \partial_t \Omega + c_2 (\Omega \cdot \nabla_x) \Omega \right) + 
  \lambda  \,  (\mbox{Id} - \Omega \otimes \Omega) \nabla_x \rho = 0,\label{eq:macro_omega}\\
  &&|\Omega | = 1,\label{eq:macro_constraint}  
\end{eqnarray}
where $c_1$, $c_2$ and $\ld$ are some constants depending on the noise parameter $d$. The
expressions of $c_1$, $c_2$ and $\ld$ are given in appendix \ref{sec:coeff}. By contrast with the
standard Euler system, the two convection coefficients $c_1$ and $c_2$ are different. The
other specificity of this model is the constraint $|\Omega|=1$. The operator $(\mbox{Id} -
\Omega \otimes \Omega)$ ensures that this constraint is propagated provided that it is
true at the initial time. The passage from (\ref{Cont_pos_bis})-(\ref{Cont_orient_bis}) to
(\ref{eq:macro_rho})-(\ref{eq:macro_omega})-(\ref{eq:macro_constraint}) is detailed in
\cite{degond2007ml}. We note that vortex configurations are special stationary solutions
of this model in two dimensions (see appendix \ref{sec:moulin}). Up to our knowledge, this is the first
swarming model which have such analytical solutions.

\section{The Macroscopic Vicsek model}
\label{sec:macro_equa}
\setcounter{equation}{0}

\subsection{Theoretical analysis of the macroscopic model}
\label{sec:th_an}

To study model (\ref{eq:macro_rho})-(\ref{eq:macro_omega})-(\ref{eq:macro_constraint}), we first use the rescaling $x' = x/c_1$. Then
equations (\ref{eq:macro_rho})-(\ref{eq:macro_omega})-(\ref{eq:macro_constraint}) are written:
\begin{eqnarray}
  &&\ds \partial_t \rho + \nabla_{x'} \cdot (\rho \Omega)  = 0,\label{eq:macro2_rho}\\
  &&\ds \rho \, \left( \partial_t \Omega + c'(\Omega \cdot \nabla_{x'}) \Omega \right)
  +  \lambda'  \,  (\mbox{Id} - \Omega \otimes \Omega) \nabla_{x'} \rho = 0,\label{eq:macro2_omega}\\
  &&|\Omega | = 1,\label{eq:macro2_constraint}    
\end{eqnarray}
with $c' = c_2/c_1$ and $\lambda' = \lambda/c_1$. In the sequel, we drop the primes for
clarity. We refer to the appendix \ref{sec:coeff} for the computation of $c$ and $\lambda$
and we just mention that we have (see figure \ref{fig:div_c1}):
\begin{equation}
  \label{eq:coeff}
  \frac{1}{2}<c<1 \quad \text{ and }\quad \ld>0,\quad \mbox{for all } d>0. 
\end{equation}

In two dimensions, we can use a parameterization of $\Omega$ in polar
coordinates: $\Omega=(\cos \theta, \sin
\theta)^T$. Therefore, equations (\ref{eq:macro2_rho})-(\ref{eq:macro2_omega}) can be rewritten as:
\begin{eqnarray}
  && \ds \partial_t \rho + \partial_x \,(\rho \cos \theta) + \partial_y\,( \rho \sin \theta) \;\; =
  \;\; 0, \label{eq:2D_theta_rho} \\
  && \ds \partial_t \theta + c\cos \theta \partial_x \theta + c \sin \theta \partial_y \theta +
  \ld \left( -\frac{\sin \theta}{\rho} \partial_x \rho + \frac{\cos \theta}{\rho} \partial_y
    \rho \right) \;\; = \;\; 0.\label{eq:2D_theta_omega}
\end{eqnarray}
In this section, we suppose that $\rho$ and $\theta$ are independent of $y$ meaning that we are
looking at waves which propagate in the x-direction.  Under this assumption, the system
reads:
\begin{equation}
  \label{eq:2D_x_nc}
  \partial_t \left(
    \begin{array}{c}
      \rho \\
      \theta
    \end{array}
  \right) + A(\rho,\theta) \, \partial_x \left(
    \begin{array}{c}
      \rho \\
      \theta
    \end{array}
  \right) = 0,\end{equation}
with 
\begin{equation}
  \label{eq:A_1}
  A(\rho,\theta) = \left[
    \begin{array}{cc}
      \cos \theta & -\rho \sin \theta \\
      -\frac{\ld \sin \theta}{\rho} & c \cos \theta
    \end{array}
  \right].
\end{equation}
The characteristic velocities of this system are given by
\begin{equation}
  \label{eq:eigenvalues}
  \gamma_{1,2} = \frac{1}{2} \left[ (c+1)\cos \theta \pm \sqrt{(c-1)^2 \cos^2 \theta
      + 4 \ld \sin^2 \theta}\right]
\end{equation}
with $\gamma_{1}<\gamma_{2}$. Therefore, the system is strictly \emph{hyperbolic}. A possible choice of right eigenvectors is
\begin{equation}
  \label{eq:eigenvector}
  \vec{r}_1 = \left( \begin{array}{c}
      \rho \sin \theta \\
      \cos \theta - \gamma_1
    \end{array} \right) \qquad , \qquad \vec{r}_2 = \left( \begin{array}{c}
      c \cos \theta - \gamma_2\\
      \frac{\ld \sin \theta}{\rho}
    \end{array} \right) .
\end{equation}
The two fields are genuinely nonlinear except at $\theta =0$, $\theta =\pi$ and at the extrema values of
$\gamma_p$ which satisfy:
\begin{displaymath}
  \tan^2 \theta = \frac{1}{4\ld}\, \left[ \frac{((c-1)^2-4\ld)^2}{(c+1)^2} - (c-1)^2 \right].
\end{displaymath}

\begin{figure}[ht]
  \centering
  \includegraphics[scale=.8]{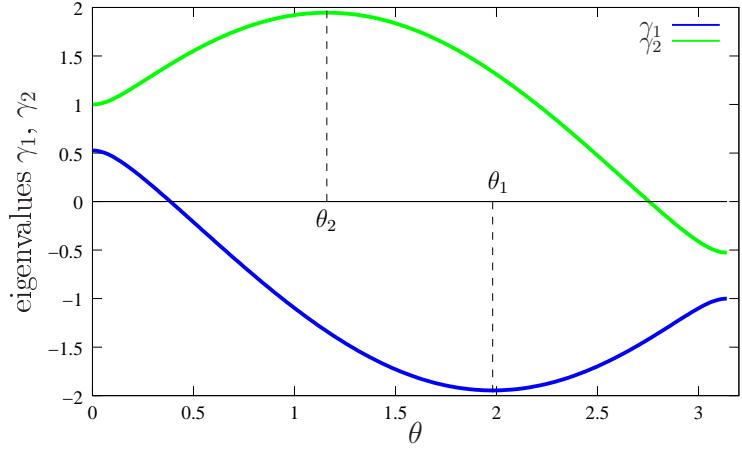}
  \caption{The two eigenvalues $\gamma_1$ and $\gamma_2$ depending on $\theta$ ($d=1$ in
    this graph). For each curve, there exists a unique extremum ($\theta_1$ and
    $\theta_2$) which corresponds to a degeneracy of the system.}
  \label{fig:vp}
\end{figure}

\noindent The Riemann invariants of the system (\ref{eq:2D_x_nc}) are given by:
\begin{eqnarray}
  \label{eq:w1}
  I_1(\rho,\theta) &=& \log \rho - \int_{\theta_0}^\theta \frac{\sin s}{\cos s -
    \gamma_1(s)}\,ds, \\
  \label{eq:w2}
  I_2(\rho,\theta) &=& \log \rho - \int_{\theta_0}^\theta \frac{c\cos s -
    \gamma_2(s)}{\ld \sin s}\,ds.
\end{eqnarray}
The integral curve $w_1$ and $w_2$ starting from $(\rho_l,\theta_l)$ are given by:
\begin{eqnarray}
  \label{eq:u1}
  \rho_1(\theta) &=& \rho_l\, \exp \left(\int_{\theta_0}^\xi \frac{\sin s}{\cos s -
      \gamma_1(s)}\,ds \right), \\
  \label{eq:u2}
  \rho_2(\theta) &=& \rho_l\, \exp \left(\int_{\theta_0}^\xi \frac{c\cos s -
      \gamma_2(s)}{\ld \sin s}\,ds \right).
\end{eqnarray}
These are the rarefaction curves. To select the physically admissible rarefaction curve,
we remark that $\gamma_{p}$ must grow from the left to right states. The proofs of these
elementary facts are omitted. The quantities $\gamma_{1,2}$ as functions of $\theta$ are
depicted in figure \ref{fig:vp}.

\subsection{A conservative form of the MV model in dimension 1}
\label{sec:cons_shock}

For non-conservative systems, shock waves are not uniquely defined
\cite{leveque2002fvm,Lefloch1988ews}. However, in the present case, a conservative
formulation of the system can be found in dimension 1. Indeed, it is an easy matter to see
that, if $\sin \theta \neq 0$, system (\ref{eq:2D_x_nc}) can be rewritten in conservative
form:
\begin{equation}
  \label{eq:1Dcons_x}
  \partial_t \left(
    \begin{array}{c}
      \rho \\
      f_1(\theta)
    \end{array}
  \right) + \partial_x \left(
    \begin{array}{c}
      \rho \cos \theta \\
      cf_2(\theta) - \ld \log(\rho)
    \end{array}
  \right)  \;\; = \;\; 0,
\end{equation}
with:
\begin{eqnarray}
  f_1(\theta) &=& \log \left|\tan \frac{\theta}{2} \right| = \log \left|\frac{\sin
      \theta }{\cos \theta + 1} \right|, \label{eq:f_1}\\
  f_2(\theta) &=& \log \left|\sin \theta \right|. \label{eq:f_2}
\end{eqnarray}
However, the functions $f_1$ and $f_2$ are singular when $\sin \theta=0$ which means
that the conservative form is only valid as long as $\theta$ stays away from $\theta =
0$.

The conservative form (\ref{eq:1Dcons_x}) leads to the following 
Rankine-Hugoniot conditions for shock waves: two states $(\rho_l,\theta_l)$ and
$(\rho_r,\theta_r)$ are connected by a shock wave traveling at a constant speed $s$ if
they satisfy:
\begin{equation}
  \label{eq:RH}
  s \left(
    \begin{array}{c}
      \rho_r - \rho_l \\
      f_1(\theta_r) - f_1(\theta_l)
    \end{array}
  \right) = \left(
    \begin{array}{c}
      \rho_r \cos \theta_r - \rho_l \cos \theta_l \\
      cf_2(\theta_r) - cf_2(\theta_l) - \ld \log \rho_r + \ld \log \rho_l
    \end{array}
  \right).
\end{equation}

We can combine the two equations of the system (\ref{eq:RH}) to eliminate the constant
$s$ and this leads to the following expression of the shock curve:
\begin{eqnarray}
  \label{eq:hugo}
  && (\rho_r\ - \rho_l)(c f_2(\theta_r) - cf_2(\theta_l) - \ld \log \rho_r + \ld \log
  \rho_l) \qquad \qquad \qquad \qquad \qquad \qquad\nb  \\
  && \qquad \qquad \qquad \qquad \qquad = (\rho_r \cos \theta_r - \rho_l \cos \theta_l) (f_1(\theta_r) -
  f_1(\theta_l)).
\end{eqnarray}
This equation must be numerically solved. The entropic part of the shock curve is
determined by the requirement that $\gamma_p$ must satisfy the Lax entropy condition.  In
figure \ref{fig:U_lmr}, we give an example of a solution of a Riemann problem obtained by
computing the intersection of the shock and rarefaction curves.

\begin{figure}[ht]
  \centering
  \includegraphics[scale=.85]{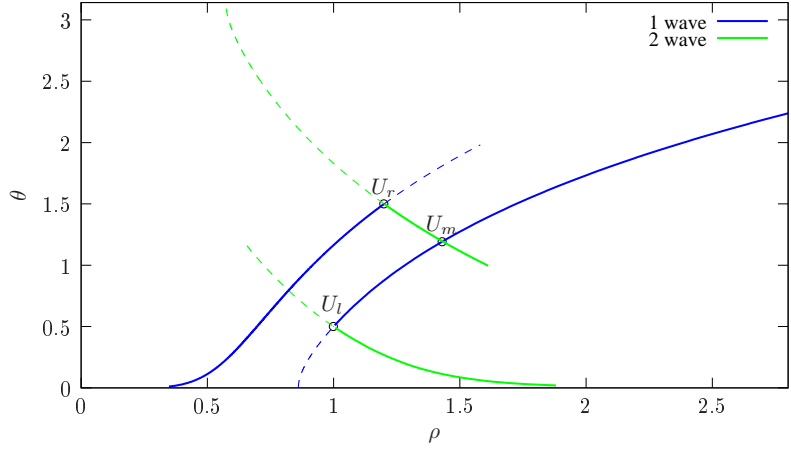}
  \caption{A solution of the Riemann problem with left and right states $U_l$ and $U_r$
    (solid line for shock waves and dotted line for rarefaction waves). In this example,
    the solution is given by two shock waves.}
  \label{fig:U_lmr}
\end{figure}

\subsection{The MV model as the relaxation limit of a conservative system}

We are going to prove that the MV model
(\ref{eq:macro2_rho})-(\ref{eq:macro2_omega})-(\ref{eq:macro2_constraint}) can be seen as
the relaxation limit of a conservative hyperbolic model with a relaxation term.  This link
will be used later to build a new numerical scheme. More precisely, we introduce the
relaxation model:
\begin{eqnarray}
  && \ds \partial_t \rho^{\ep} + \nabla_{x} \cdot (\rho^{\ep} \Omega^{\ep})  = 0,\label{eq:UMV_rho}  \\
  &&  \ds \partial_t \left(\rho^{\ep} \Omega^{\ep}\right) + c \nabla_{x} \cdot
  \left(\rho^{\ep} \Omega^{\ep} \otimes \Omega^{\ep} \right)  +  \lambda  \, \nabla_{x}
  \rho^{\ep} = \frac{\rho^\ep}{\ep} (1-|\Omega^{\ep}|^2)\Omega^{\ep}.\label{eq:UMV_omega}
\end{eqnarray}
In this model, the constraint $|\Omega|=1$ is replaced by a relaxation operator. Formally,
in the limit $\ep\rightarrow 0$, we recover the constraint $|\Omega|=1$.
\begin{proposition}
  \label{ppo:UMV_MV}
  The relaxation model (\ref{eq:UMV_rho})-(\ref{eq:UMV_rho}) converges to the MV model
  (\ref{eq:macro2_rho})-(\ref{eq:macro2_omega})-(\ref{eq:macro2_constraint}) as $\ep$ goes
  to zero.
\end{proposition}
\begin{proof_formal}
  We define $R^\ep=\rho^\ep (1-|\Omega^{\ep}|^2)\Omega^{\ep}$. Suppose that as $\ep$ goes to zero:
  \begin{equation}
    \label{eq:UMV_cv}
    \rho^{\ep} \stackrel{\ep \rightarrow 0}{\longrightarrow} \rho^0 \quad ,
    \quad \Omega^{\ep} \stackrel{\ep \rightarrow 0}{\longrightarrow}\Omega^0.
  \end{equation}
  Then $R^\ep \stackrel{\ep \rightarrow 0}{\longrightarrow} 0$, which generically implies
  that $|\Omega^{0}|^2=1$ (except where $\rho^{0}\Omega^{0} = 1$ which one assumes to be a
  negligible set). Therefore, we have:
  \begin{equation}
    \label{eq:om_0}
    \partial_t \Omega^0 \cdot \Omega^0=0 \quad , \quad (\Omega^0 \cdot \nabla_x)\Omega^0 \,\cdot \Omega^0 =0.
  \end{equation}
  Then since $R^\ep \times \Omega^\ep=0$, we have:
  \begin{displaymath}
    (\partial_t \left(\rho^{\ep} \Omega^{\ep}\right) + c \nabla_{x} \cdot
    \left(\rho^{\ep} \Omega^{\ep} \otimes \Omega^{\ep} \right)  +  \lambda  \, \nabla_{x}
    \rho^{\ep})\times \Omega^\ep = 0.
  \end{displaymath}
  and consequently when $\ep\rightarrow 0$:
  \begin{equation}
    \label{eq:om_0_alpha}
    \partial_t \left(\rho^{0} \Omega^{0}\right) + c \nabla_{x} \cdot
    \left(\rho^{0} \Omega^{0} \otimes \Omega^{0} \right)  +  \lambda  \, \nabla_{x}
    \rho^{0} = \alpha \Omega^0, 
  \end{equation}
  for a real number $\alpha$ to be determined. Taking the scalar product of
  (\ref{eq:om_0_alpha}) with $\Omega^0$ and using (\ref{eq:om_0}), we find:
  \begin{equation*}
    \alpha = \partial_t \rho^{0} + c \nabla_{x} \cdot (\rho^{0} \Omega^{0}) + \ld \nabla_{x}
    \rho^{0} \cdot \Omega^0.
  \end{equation*}
  Using the conservation of mass ($\partial_t \rho^0 = -\nabla_x \cdot (\rho^0
  \Omega^0)$), we finally have:
  \begin{displaymath}
    \alpha = (c-1) \nabla_{x} \cdot (\rho^{0} \Omega^{0}) + \ld \nabla_{x}
    \rho^{0} \cdot \Omega^0.
  \end{displaymath}
  Therefore, the relaxation term satisfies:
  \begin{displaymath}
    \frac{1}{\ep} R^\ep = [(c-1) \nabla_{x} \cdot (\rho^{0} \Omega^{0}) + \ld \nabla_{x}
    \rho^{0} \cdot \Omega^0] \Omega^0 + O(\ep).
  \end{displaymath}
  Inserting in (\ref{eq:UMV_rho})-(\ref{eq:UMV_omega}) and taking the limit $\ep
  \rightarrow 0$, we recover the MV model (\ref{eq:macro2_rho})-(\ref{eq:macro2_omega}) at
  the first order in $\ep$.

\end{proof_formal}  

\begin{remark}
  As for the MV model, we can also analyze the hyperbolicity of the left hand side of
  (\ref{eq:UMV_rho})-(\ref{eq:UMV_omega}). The eigenvalues are given by:
  \begin{displaymath}
    \gamma_1 = c u -\sqrt{\Delta} \quad , \quad \gamma_2 = c u \quad , \quad \gamma_3 = c
    u + \sqrt{\Delta},
  \end{displaymath}
  where $u$ denotes the $x$-coordinate of $\Omega$ and $\Delta = \ld - (c-c^2) u^2$. The
  system is hyperbolic if and only if $|u|< \sqrt{\frac{\ld}{c-c^2}}$. As we can see in
  figure \ref{fig:hyp_UMV2}, for $u^2=1$, $\Delta$ is positive for any values of the noise
  parameter $d$. In particular, this implies that the relaxation model is hyperbolic for
  every $|u| \leq 1$.
\end{remark}

\begin{figure}[ht]
  \centering
  \includegraphics[scale=.7]{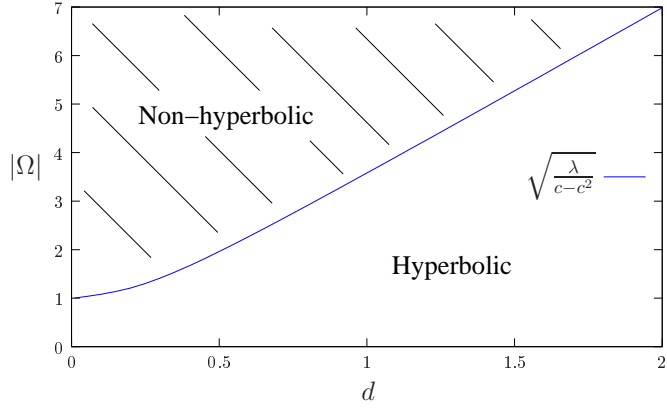}
  \caption{The quantity $\sqrt{\ld/(c-c^2)}$ depending on $d$. The relaxation model (\ref{eq:UMV_rho})-(\ref{eq:UMV_rho}) is
    hyperbolic when the speed $|\Omega|$ is below this curve. At the limit $\ep \rightarrow 0$,
    $|\Omega^\ep | \rightarrow 1$ and therefore the relaxation model is hyperbolic for any $d$ in this
    limit.}
  \label{fig:hyp_UMV2}
\end{figure}

\section{Numerical simulations of the MV model}
\label{sec:num_simu_cva_macro}

\subsection{Numerical schemes}
\label{sec:diff_scheme_non_cons}

We propose four different numerical schemes to solve the MV model. The first two schemes
originate from the discussions of the previous section, the two other one are based on the
non-conservative form of the MV model. \\
We use the following notations: we fix a uniform stencil $(x_i)_i$ (with
$|x_{i+1}-x_i|=\Delta x$) and a time step $\Delta t$. We denote by $U_i^n = (\rho_i^n,\,
\theta_i^n)$ the value of the mass and flux direction at the position $x_i$ and at time
$n\Delta t$.

\subsubsection{The conservative scheme}

Here we use the conservative form of the MV model (\ref{eq:1Dcons_x}):
\begin{equation}
  \label{eq:V}
  \partial_t V + \partial_x F(V) = 0,
\end{equation}
with $V=(\rho,f_1(\theta))^T$ and $F(V) = (\rho \cos \theta,\, cf_2(\theta) - \ld
\log(\rho))^T$. We use a Roe method to discretize this equation:
\begin{equation}
  \label{eq:cons}
  \frac{V_i^{n+1} - V_i^n}{\Delta t} + \frac{\widehat{F}_{i+} - \widehat{F}_{i-}}{\Delta x} = 0,
\end{equation}
where the intermediate flux $\widehat{F}_{i+}$ is given by:
\begin{equation}
  \label{eq:flux_c}
  \widehat{F}_{i+} = \frac{F(V_i) + F(V_{i+1})}{2}\, -
  \left|\mathcal{A}(\overline{V}_{i+})\right| \frac{V_{i+1} - V_{i}}{2},
\end{equation}
and $\mathcal{A}$ is the Jacobian of the flux $F$:
\begin{equation}
  \label{eq:A_c}
  \mathcal{A}(V) = DF(V) = \left[ \begin{array}{cc}
      \cos \theta & -\rho \sin^2 \theta \\
      -\frac{\ld}{\rho} & c \cos \theta
    \end{array}  \right] 
\end{equation}
calculated at the mean value $\overline{V}_{i+} = \frac{V_i+V_{i+1}}{2}.$

As mentioned earlier, the conservative form is only valid when $\theta$ does not cross a
singularity $\theta=0$ or $\theta = \pi$ (i.e. $\sin \theta=0$). Nevertheless, numerically
we can still use the formulation (\ref{eq:cons}) when $\theta$ changes sign. Moreover,
since $f_1$ is an even function, this only gives $|\theta^{n+1}|$.  To determine the sign
of $\theta^{n+1}$, we use an auxiliary value $\widehat{\theta}$ which we update with the
upwind scheme (\ref{eq:upwind}). The sign of $\theta$ is then determined using the sign of
$\widehat{\theta}$.

\subsubsection{The splitting method}
\label{sec:splitting_method}

The next scheme uses the relaxation model (\ref{eq:UMV_rho})-(\ref{eq:UMV_omega}). The
idea is to split the relaxation model in two parts, first the conservative part:
\begin{equation}
  \label{eq:UMV_cons}
  \begin{array}{ll}
    & \ds \partial_t \rho + \nabla_{x} \cdot (\rho \Omega)  = 0,  \\
    &  \ds \partial_t \left(\rho \Omega\right) + c \nabla_{x} \cdot
    \left(\rho \Omega \otimes \Omega \right)  +  \lambda  \, \nabla_{x}
    \rho = 0.
  \end{array}
\end{equation}
and then the relaxation part:
\begin{equation}
  \label{eq:UMV_relax}
  \begin{array}{ccl}
    \ds \partial_t \rho  & = & 0,  \\
    \ds \partial_t \left(\rho \Omega\right) & = & \ds \frac{\rho}{\ep} (1-|\Omega|^2)\Omega.
  \end{array}
\end{equation}
This system reduces to: $\partial_t \Omega = \frac{1}{\ep} (1-|\Omega|^2)\Omega.$ Since
this equation only changes the vector field $\Omega$ in norm (i.e. $\partial_t \Omega
\cdot \Omega^\perp =0$), we can once again reduce this equation to:
\begin{equation}
  \label{eq:omega2}
  \frac{1}{2} \partial_t |\Omega|^2 = \frac{1}{\ep} (1-|\Omega|^2)|\Omega|^2.
\end{equation}
Equation (\ref{eq:omega2}) can be explicitly solved: $|\Omega|^2 \!=\!
(1 + C_0\,\expo^{-2/\ep\,t})^{-1}$
% \begin{equation}
%   \label{eq:sol_omega2}
%   |\Omega|^2 = \frac{1}{1+C_0\,\expo^{-2/\ep\,t}},
% \end{equation}
with $C_0 \!=\! \left(\frac{1}{|\Omega_0|^2} -1\right)$. We indeed take the limit $\ep \rightarrow 0$ of this
expression and replace the relation term by a mere normalization:
$\Omega\rightarrow \Omega/|\Omega|$ .

The conservative part is solved by a Roe method with a Roe matrix computed following
\cite{leveque2002fvm} page 156.

\subsubsection{Non-conservative schemes}

We present two other numerical schemes based on the
non-conservative formulation of the MV model.

\medskip

\noindent {\bf (i) Upwind scheme}

The method consists to update the value of $U_i^n$ with the formula:
\begin{equation}
  \label{eq:upwind}
  \frac{U_i^{n+1} - U_i^n}{\Delta t} + A^+ \left( \frac{U_i^n - U_{i-1}^n}{\Delta t}
  \right) + A^- \left( \frac{U_{i+1}^n - U_i^n}{\Delta t} \right) = 0,
\end{equation}
where $A^+$ and $A^-$ are (respectively) the positive and negative part of $A$, defined such that
$A = A^+ - A^-$ and $|A| = A^+ + A^-$ and $A^+,\,A^-$ are computed using an explicit
diagonalization of $A$.

\noindent {\bf (ii) Semi-conservative scheme}

One of the problem with the upwind scheme is that it does not conserve the total mass ($\int_x
\rho(x)\,dx$). In order to keep this quantity constant in time, we use
the equation of conservation of mass (\ref{eq:macro2_rho}) in a conservative form:
\begin{equation}
  \label{eq:rho_c}
  \partial_t \rho + \partial_x H(\rho,\theta) = 0,
\end{equation}
with $H(\rho,\theta) = \rho \cos \theta$. Therefore, a conservative numerical scheme
associated with this equation would be:
\begin{equation}
  \label{eq:rho_c_num}
  \frac{\rho_i^{n+1} -\rho_i^n}{\Delta t} + \frac{\widehat{H}_{i+1/2} - \widehat{H}_{i-1/2}}{\Delta x} = 0,
\end{equation}
where $\widehat{H}_{i+}$ is the numerical estimation of the flux $H$ at the interface
between $x_i$ and $x_{i+1}$. To estimate numerically this flux, we use the following
formula with $U_{i}=(\rho_{i},\theta_{i})$:
\begin{equation}
  \label{eq:H_num}
  \widehat{H}_{i+1/2} = H(U_{i+1/2}) - \left|A\right|_{\rho}
  \left(\frac{U_{i+1}^n-U_i^{n}}{2}\right),
\end{equation}
where the intermediate value is given by $U_{i+1/2}=\frac{U_{i}^n + U_{i+1}^{n}}{2}$ and
$\left|A\right|_{\rho}$ is the first line of the absolute value of $A$.

For the estimation of the angle $\theta$, we keep the same scheme as for the upwind
scheme. This numerical scheme uses one conservative equation (for the mass $\rho$) and a
non-conservative equation (for the angle $\theta$). It is thus referred to as the
\textit{semi-conservative} scheme.

\subsection{Numerical simulations}

To compare the various numerical schemes, we use a Riemann problem as initial
condition. We choose solutions which consist of a rarefaction wave (figure
\ref{fig:rare11}) or a single shock wave (figures \ref{fig:choc1}-\ref{fig:choc111}).
%% or composed of a rarefaction wave and a shock wave (figure \ref{fig:rare_choc}).

We take the following parameters: $d = 1$, the length of the domain is
$10$ units and the discontinuity for the Riemann problem is at $x=5$ (the middle of
the domain). The simulations are run during two time units with a time step $\Delta
t=2.10^{-2}$ and a space step $\Delta x = 5.10^{-2}$. For these values, the Courant number
($C_n$) is $0.778$. We use homogeneous Neumann conditions as boundary conditions.

For the rarefaction wave, we take:
\begin{equation}
  \label{eq:riemann_rare}
  (\rho_l,\theta_l)=(2,1.7) \quad , \quad (\rho_r,\theta_r) = (1.12,0.60).
\end{equation}
All the numerical schemes capture well the theoretical solution (see figure \ref{fig:rare11}).

For the shock wave, we choose:
\begin{equation}
  \label{eq:riemann_shock}
  (\rho_l,\theta_l)=(1,1.05) \quad , \quad (\rho_r,\theta_r) = (1.432\,,\,1.7).
\end{equation}
For these values, the shock speed computed with (\ref{eq:hugo}) is $s=-1.585$. The results
of the numerical simulations using the four schemes are given in figure
\ref{fig:choc1}. The numerical solutions are in accordance with the theoretical solution
given by the conservative formulation for all the numerical schemes. Nevertheless, the
conservative scheme is in better accordance with this solution. For the
other schemes, the shock speed differs slightly.\\
A second example of a shock wave is computed using the following initial condition:
\begin{equation}
  \label{eq:riemann_shock2}
  (\rho_l,\theta_l)=(1,0.314) \quad , \quad (\rho_r,\theta_r) = (2,1.54).
\end{equation}
The solutions given by the $4$ numerical
schemes are very different. Only the conservative method is in agreement with the
solution given by the conservative formulation. But the conservative formulation is not
necessary the right one.  Indeed, in the next section, particle simulations show that the right
solution is not given by the conservative formulation but rather by the splitting method.

\vspace{1cm}

\begin{figure}[ht]
  \centering
  \includegraphics[scale=1]{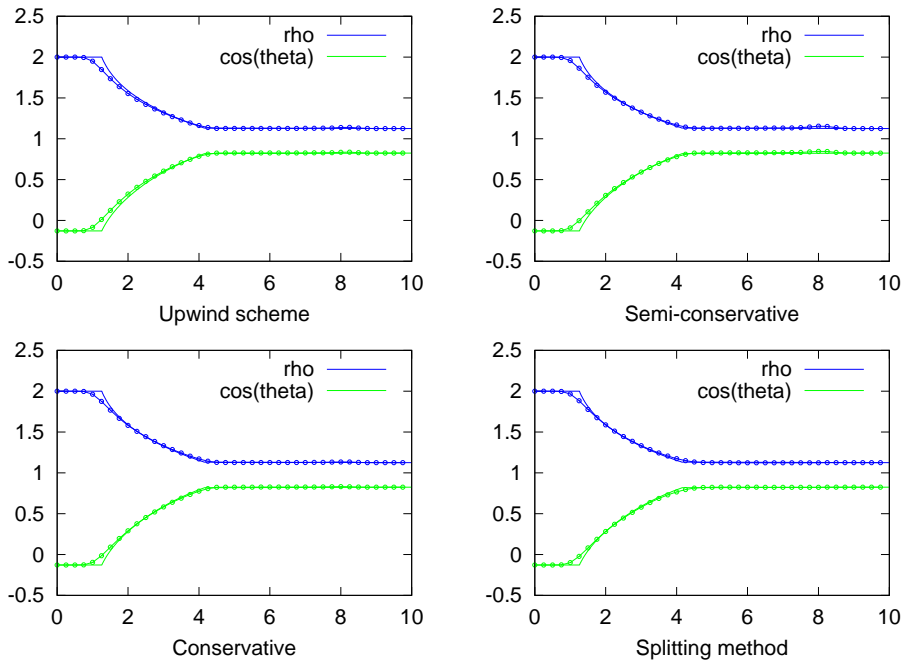}
  \caption{The theoretical solution of the Riemann problem (\ref{eq:riemann_rare}) given
    by a rarefaction curve (solid line) and the numerical solutions (points), $\rho$
    (blue) and $\cos \theta$ (green) as functions of space. The simulations are run during
    2 time units, with a time step $\Delta t= 2.10^{-2}$ and a space step $\Delta x =
    5.10^{-2}$ (CFL=$.778$).}
  \label{fig:rare11}
\end{figure}

\begin{figure}[p]
  \centering
  \includegraphics[scale=1]{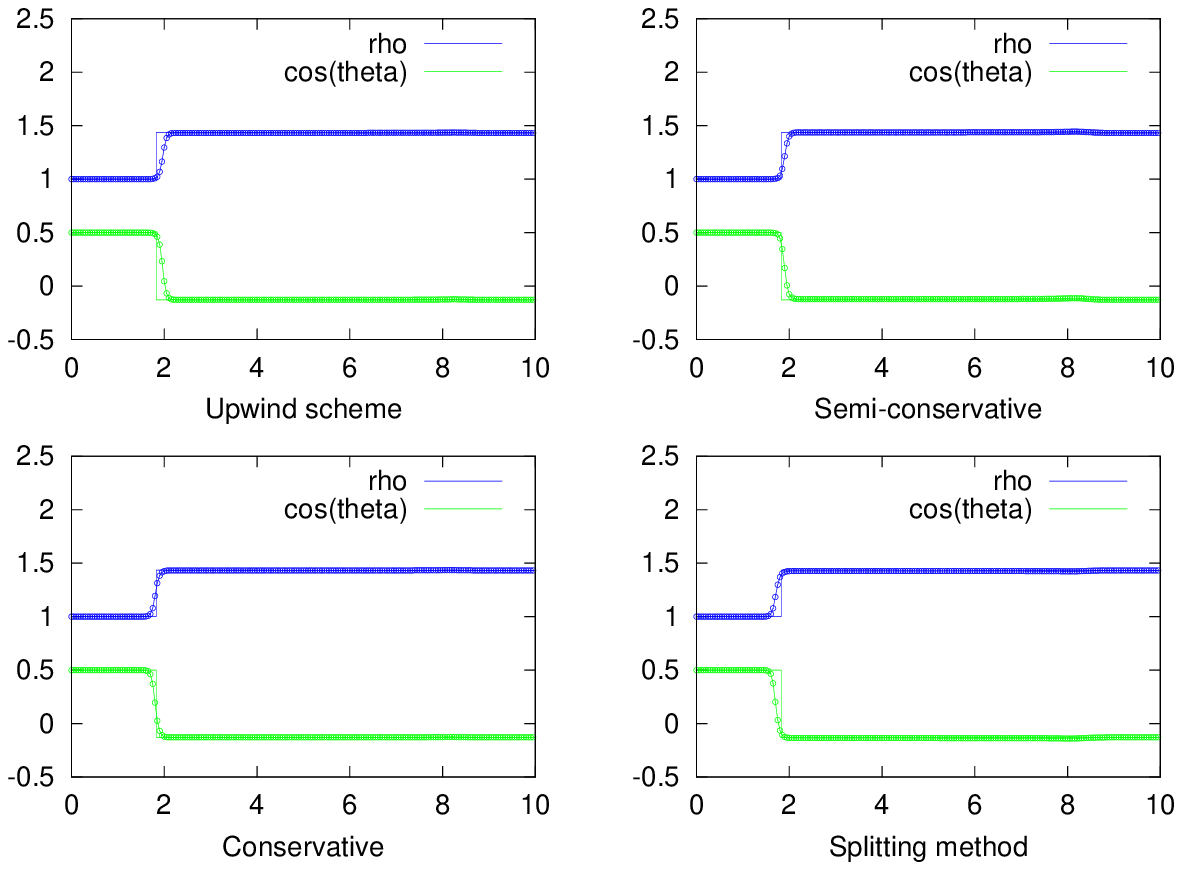}
  \caption{Theoretical and numerical solutions of the Riemann problem (\ref{eq:riemann_shock})}
  \label{fig:choc1}
\end{figure}

\begin{figure}[p]
  \centering
  \includegraphics[scale=1]{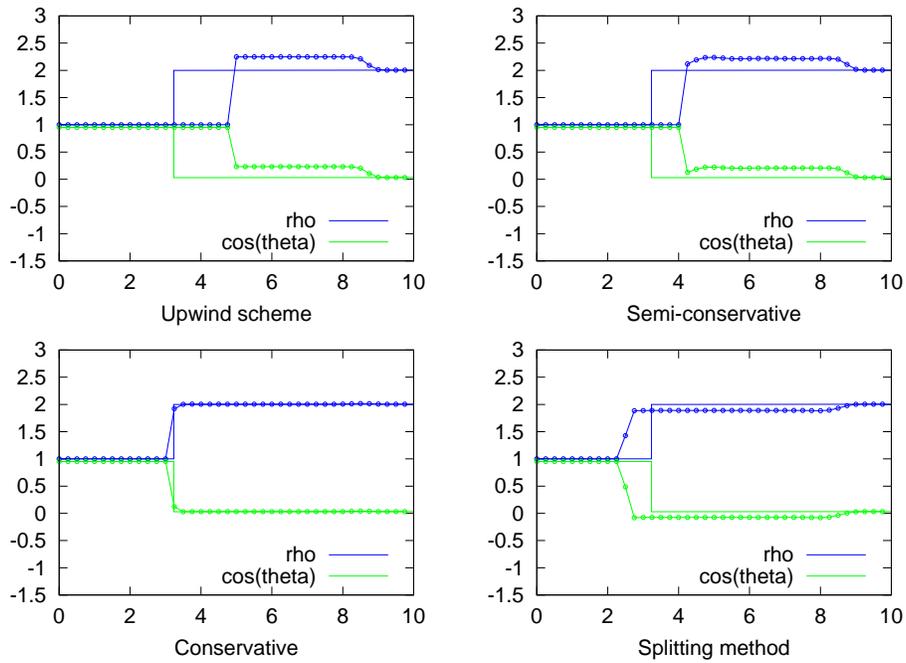}
  \caption{Theoretical and numerical solutions of the Riemann problem (\ref{eq:riemann_shock2})}
  \label{fig:choc111}
\end{figure}

\clearpage

\section[The microscopic versus macroscopic Vicsek models]{The microscopic versus macroscopic Vicsek \\ models}
\label{sec:micro_cva}
\setcounter{equation}{0}

\subsection{Local equilibrium}

In this part, we would like to validate the macroscopic Vicsek model by the simulation of
the microscopic Vicsek. The macroscopic model relies on the fact that the particle
distribution function is at local equilibrium given by a Von Mises distribution $M_\Omega$
(see \cite{degond2007ml}):
\begin{equation}
  \label{eq:M_von_mises}
  M_{\Omega}(\omega) = C \exp\left(\frac{\omega \cdot \Omega}{d}\right)
\end{equation}
where $C$ is set by the normalization condition\footnote{explicitly given by
  $C^{-1}=2\pi\,I_0(d^{-1})$ where $I_0$ is the modified Bessel function of order 0}. The
goal of this section is to show numerically that the particle distribution of the
microscopic Vicsek model is close in certain regimes to this Von Mises distribution.

To this aim, in appendix \ref{sec:num_scheme_ps} we propose a numerical scheme to solve
system (\ref{Cont_orient_bis}). The setting for our particle simulations is as follows: we
consider a square box with periodic boundary conditions. As initial condition for the
position $x_i$, we choose a uniform random distribution in space. The velocity is
initially distributed according to a uniform distribution on the unit circle.

During the simulation, we compute the empirical distribution of the velocity direction $\theta$
and of the mean velocity $\Omega$ of particles. We then compare this empirical
distribution with its theoretical distribution $M_\Omega(\theta)$ given by
(\ref{eq:M_von_mises}). In figure \ref{fig:distri_speed}, we give an example of a
comparison between the distribution of the velocity direction $\theta$ and the theoretical
distribution $M_\Omega$ predicted by the theory.

\begin{figure}[p]
  \hspace{-1cm}
  \begin{minipage}[b]{0.5\linewidth}
    \includegraphics[scale=.64]{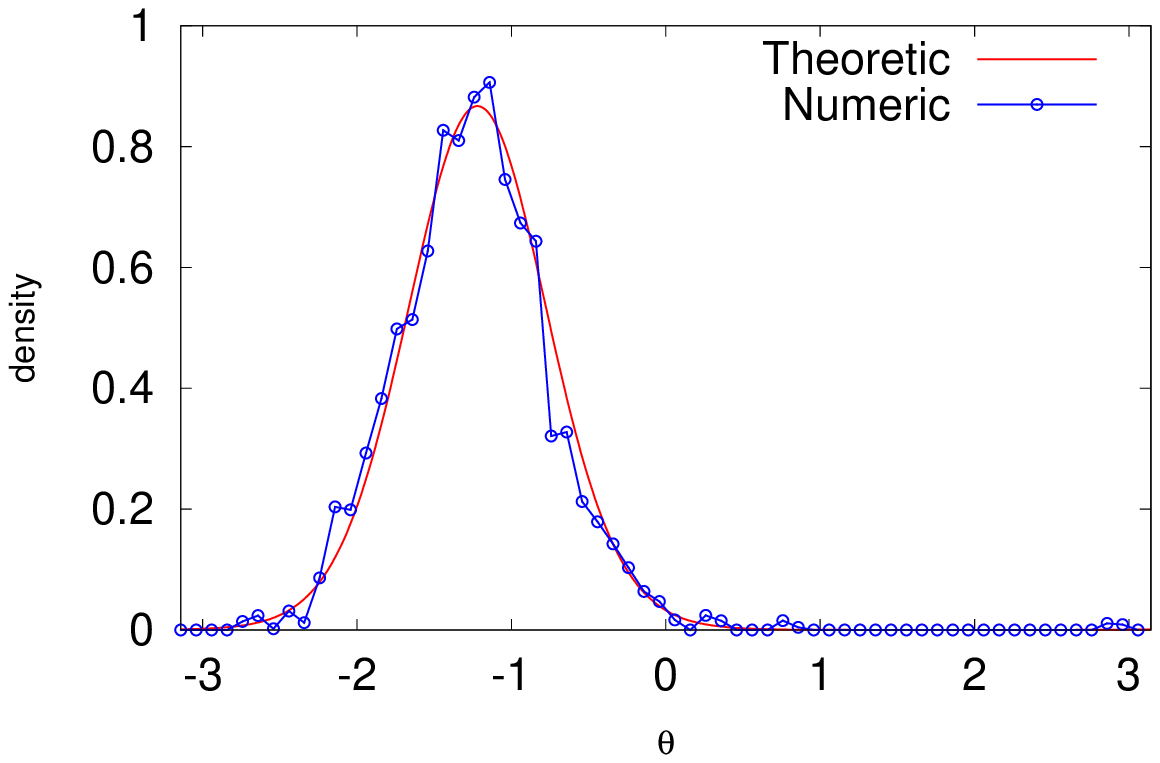}
  \end{minipage} 
  \begin{minipage}[b]{0.5\linewidth}
    \includegraphics[scale=.7]{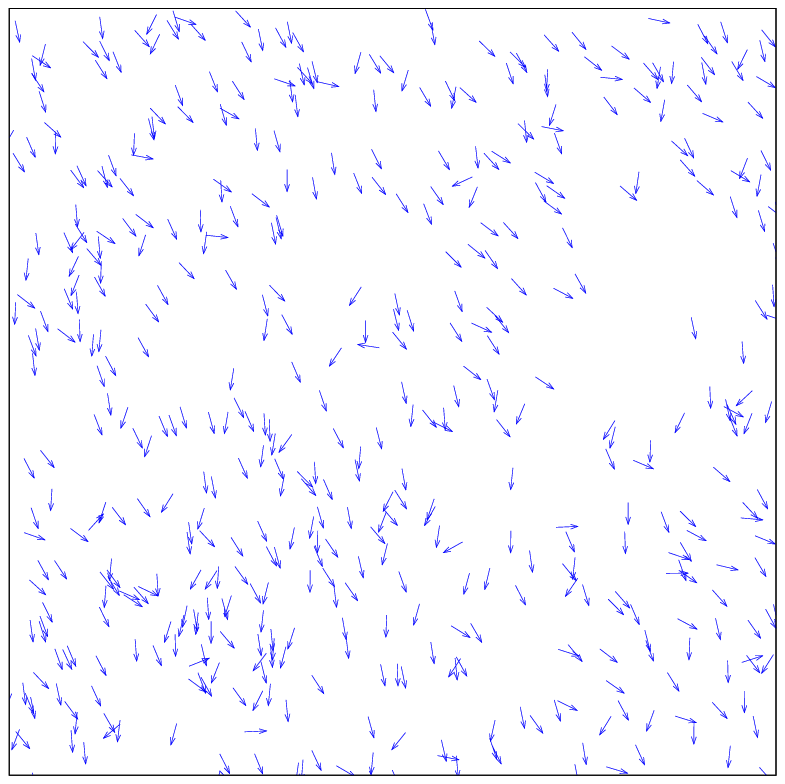}
    % \caption{Légende 2}
  \end{minipage}
  \caption{Left figure: the distribution of velocity direction $\theta$ (with $\omega =
    (\cos \theta,\,\sin \theta)$) compared with its theoretical distribution after $6$
    time units of simulation. Right figure: the corresponding particle
    simulation. Parameters of the simulation: $Lx=1$, $Ly=1$ (domain size), number of
    particles $N=500$, $\ep=1/4$, $R=.5$, $d=.2$, $\Delta t=2.10^{-3}$.}
  \label{fig:distri_speed}
\end{figure}
\begin{figure}[p]
  \centering
  \includegraphics[scale=.8]{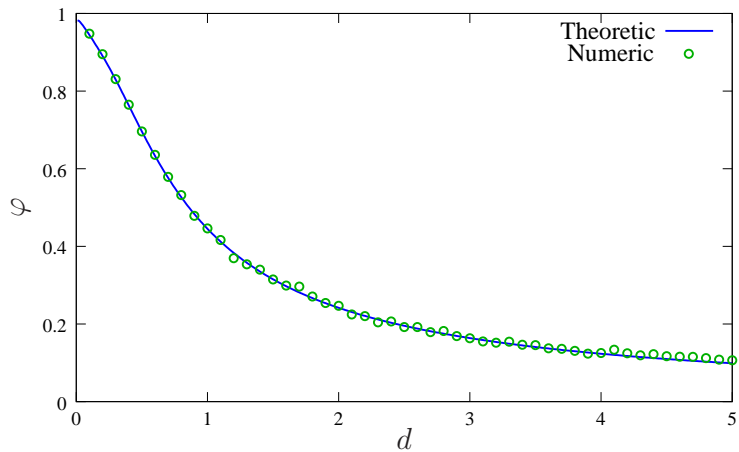}
  \caption{The mean velocity $\varphi$ (\ref{eq:total_flux}) for different values of
    $d$. Parameters of the simulation: $Lx=1$, $Ly=1$ (domain size), number of particles
    $N=200$, radius of interaction $R=.5$, $\Delta t=.02$ unit time, the simulations are
    run during $180$ unit time.}
  \label{fig:Mean_speed}
\end{figure}

Since the distribution of velocity converges, we have a theoretical value for the mean
velocity. We denote by $\varphi_N$ the mean velocity of particles and $\varphi$ the theoretical
value given by the stationary distribution:
\begin{equation}
  \label{eq:total_flux}
  \varphi_N = \frac{1}{N} \left|\sum_{k=1}^N \omega_k\right| \quad , \quad %\stackrel{\epsilon
  % \rightarrow 0}{\longrightarrow}
  \varphi = \left|\int_{\omega} \omega\,M_{\Omega}(\omega) \,d\omega \right|.
\end{equation}
At least locally in $x$, we have  that $\varphi_N \stackrel{\epsilon \rightarrow
  0}{\longrightarrow} \varphi$. In figure \ref{fig:Mean_speed}, we compare the two
distributions for different values of the noise $d$ and we can see that the two distributions are
in good agreement. We also observe a smooth transition from order ($\varphi \approx 1$) to
disorder ($\varphi <<1 $) as it has been measured in the original Vicsek model \cite{vicsek1995ntp}.

The situation is different when we look at a larger system. We still have convergence of
the velocity distribution of particles to a local equilibrium
$\rho(x)\,M_{\Omega(x)}(\omega)$, but the mean direction $\Omega(x)$ now depends on
$x$. Therefore the mean velocity of the particles in all the domain differs from the
expected theoretical value (\ref{eq:M_von_mises})-(\ref{eq:total_flux}). We illustrate
this phenomena in figure \ref{fig:Mean_speed_diff}: we fix the density of particles and we
increase the size of the box. As we can observe, the mean velocity $\varphi_N$
(\ref{eq:total_flux}) has a smaller value when the size of the box increases. This
phenomena has been previously observed in \cite{chate2008cms}. The mean velocity
$\varphi_N$ can also differ from the expected theoretical value $\varphi$
(\ref{eq:M_von_mises})-(\ref{eq:total_flux}) when the density of particles is low. In
figure \ref{fig:Mean_speed_diff2}, we fix the size of the box ($L=10$) and we increase the
density of particles (the density is given by the number of particles inside the circle of
interaction). At low density, the mean velocity $\varphi_N$ is much more smaller than the
theoretical prediction $\varphi$. But as the density of particles increases, the mean
velocity $\varphi_N$ grows (see also \cite{vicsek1995ntp}) and moreover $\varphi_N$
converges to $\varphi$. Because of that, a dense regime of particles has to be used in the
following in order to numerically compare the microscopic model with the MV model.

\begin{figure}[p]
  \centering
  \includegraphics[scale=.7]{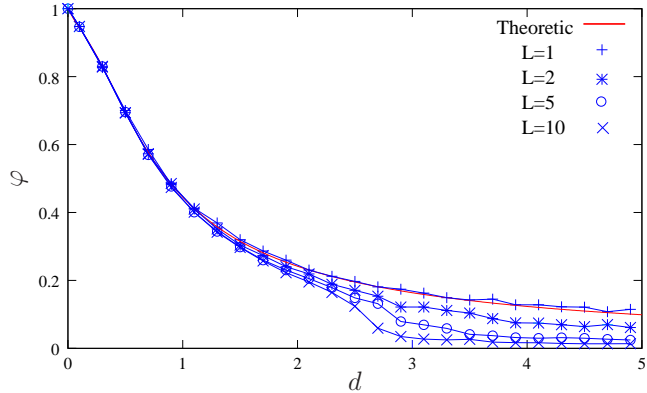}
  \caption{The mean velocity $\varphi$ (\ref{eq:total_flux}) for different values of
    $d$. We use different domain sizes and we keep the same density of particles. As the
    domain size increases, the total flux $\varphi$ decreases which means that particles
    are less aligned globally. Parameters of the simulations: $L=1,\,2,\,5,\,10$ (domain
    size), number of particles $N=200,\,800,\,5000,\,20000$, radius of interaction $R=.5$,
    $\Delta t=.02$ time units, the simulations are run during $180$ time units.}
  \label{fig:Mean_speed_diff}
\end{figure}
\begin{figure}[p]
  \centering
  \includegraphics[scale=.7]{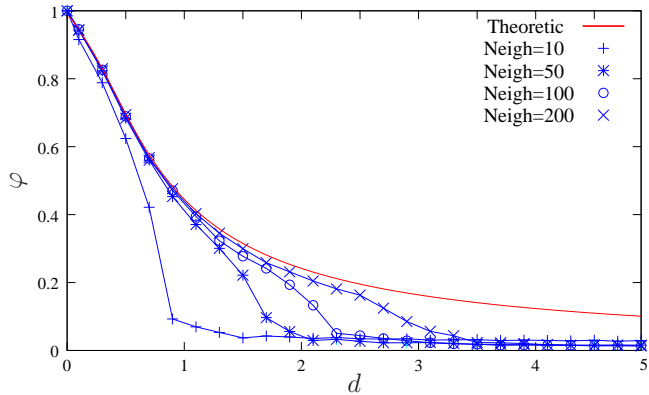}
  \caption{The mean velocity $\varphi$ (\ref{eq:total_flux}) for different values of
    $d$. We change the density of particles (given by the mean number of neighbors in
    unit of radius of interaction). When we increase the mean number of neighbors,
    particles are more aligned. Parameters of the simulations: $L=10$ (domain size),
    number of particles $N=254,\,1273,\,6366,\,12732,$ and $25464$, radius of interaction
    $R=.5$, $\Delta t=.02$ time units, the simulations are run during $180$ time units.}
  \label{fig:Mean_speed_diff2}
\end{figure}

\subsection{Microscopic versus Macroscopic dynamics}

We now compare the evolution of the two macroscopic quantities $\rho$ and $\Omega$ for the
two models. We have seen that the different schemes applied to the macroscopic equation
can give different solutions (see figure \ref{fig:choc111}). Therefore, we expect that
particle simulations will indicate what is the physically relevant solution of the macroscopic
equation.

We first briefly explain how we proceed to run the particle simulations of a Riemann
problem (see also appendix \ref{sec:num_scheme_ps}). First, we have to choose a left state
$(\rho_l,\theta_l)$, a right state $(\rho_r,\theta_r)$ and the noise parameter $d$. Then
we distribute a proportion $\frac{\rho_l}{\rho_l+\rho_r}$ of particles uniformly in the
interval $[0,5]$ and the remaining particles uniformly in the interval $[5,10]$. Then, we
generate velocity distribution $\theta$ for the particles according to the distribution
$M_{\Omega}$ (\ref{eq:M_von_mises}) with $\Omega_l =(\cos \theta_l,\sin \theta_l)^T$ on
the left side and $\Omega_r =(\cos \theta_r,\sin \theta_r)^T$ on the right side. We use
the numerical scheme given in appendix \ref{sec:num_scheme_ps} to generate particle
trajectories. To make the computation simpler, we choose periodic boundary
conditions. Therefore the number of particles is conserved. As a consequence, there are
two Riemann problems corresponding to discontinuities at $x=5$ and at $x=0$ or $10$ (which
is the same by periodicity). We use a particle-in-cell method
\cite{fehske2007cmp,hockney1988csu} to estimate the two macroscopic quantities: the
density $\rho$ and the direction of the flux $\Omega$ (which gives $\theta$). In order to
reduce the noise due to the finite number of particles, we take a mean over several
simulations to estimate the density $\rho$ and $\theta$ ($10$ simulations in our
examples).

In figure \ref{fig:exple_2D}, we show a numerical solution for the following Riemann
problem:
\begin{equation}
  \label{eq:IC_shock}
  (\rho_l,\,\theta_l) = (1,\,1.5) \quad , \quad (\rho_r,\,\theta_r) = (2,\,1.83) \quad
  , \quad d = 0.2
\end{equation}
using particle simulations and the macroscopic equation. We represent the solutions in a 
2D representation. Since the initial condition is such that the density $\rho$ and the
direction $\theta$ are independent of the $y$-direction, we only represent $\rho$ and
$\theta$ along the $x$-axis in the following figures.

In figure \ref{fig:choc_d02_particle_split}, we represent the two solutions (the particle
and the macroscopic one) with only a dependence in the $x$-direction. Three quantities are
represented: the density (blue), the flux direction $\theta$ (green) and the variance of the
angle distribution (red). The macroscopic model supposes that the variance of $\theta$
should be constant everywhere. Nevertheless, we can see that the variance is larger in
regions where the density is lower. For $\rho$ and $\theta$, we see clearly the
propagation of a shock in the middle of the domain and a rarefaction at the boundary. The
CPU time for one numerical solution at the particle level is about $140$ seconds with the
parameters given in figure \ref{fig:choc_d02_particle_split}. For the macroscopic
equation, the CPU time is about $0.1$ second which represents a cost reduction of three
orders of magnitude compared with the particle simulations. Since we have to take a
mean over many particle simulations, the cost reduction is even larger.

In figure \ref{fig:choc111_particle}, we use the same Riemann problem to set the initial
position as in figure \ref{fig:choc111} both with $d=1$ (\ref{eq:riemann_shock2}). We
use a larger domain in $x$ ($L=20$ space units) in order to avoid the effect of the periodic
boundary condition. The upwind scheme and the semi-conservative method are clearly not in
accordance with the particle simulations. Moreover, the splitting method is in better
agreement with the particle simulation since the shock speed is closer to the values given
by the particle simulations than that predicted by the conservative scheme.

Finally, our last simulations concern a contact discontinuity. We simply initialize with:
\begin{equation}
  \label{eq:theta_pm}
  (\rho_l,\theta_l) = (1,1) \quad , \quad (\rho_r,\theta_r) = (1,-1) \quad , \quad d=0.2,
\end{equation}
i.e. we reflect the angle with respect to the x-axis across the middle point $x=5$.
A natural solution for this problem is the contact discontinuity propagating
at speed $c\,cos(1)$:
\begin{equation}
  \label{eq:theta_pm_sol}
  \rho(t,x)=1 \quad , \quad \theta(t,x) = \theta_0(x - c\cos(1) t),
\end{equation}
with $\theta_0(x)=-1$ when $x<5$ and $\theta_0(x)=1$ when $x>5$. This is the solution
provided by the conservative scheme (figure \ref{fig:theta_pm}). But surprisingly, the
splitting method and the particle simulation agree on a different solution. Indeed, the
solutions given by the particles and the splitting method are in fairly good agreement
with each other, which seems to indicate that the ``physical solution'' to the contact
problem (\ref{eq:theta_pm}) is not given by the conservative formulation
(\ref{eq:theta_pm_sol}) but by a much more complex profile.  The constraint of unit speed
drastically changes the profile of the solution compared with what would be found for a
standard system of conservative laws.

%%%%%%%%% Ly=\ep  %%%%%%%%%
\begin{figure}[p]
  \hspace{-1cm}
  \begin{minipage}[b]{0.5\linewidth}
    \centering \includegraphics[scale=.7]{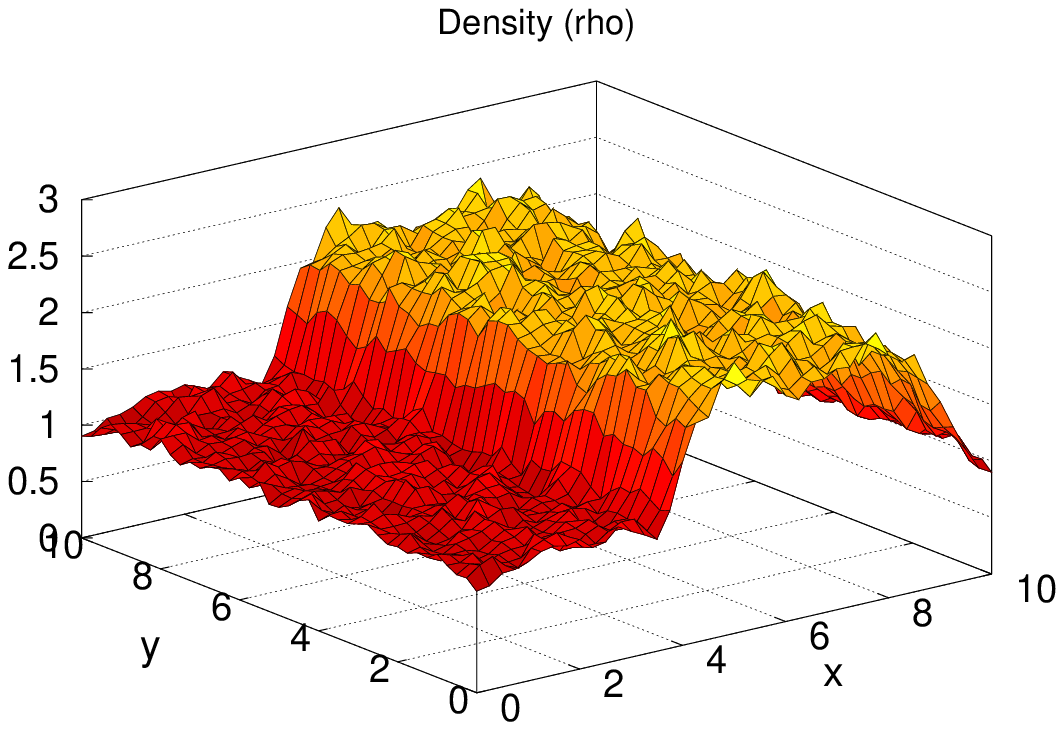}
  \end{minipage}\hfill
  \begin{minipage}[b]{0.5\linewidth}
    \centering \includegraphics[scale=.7]{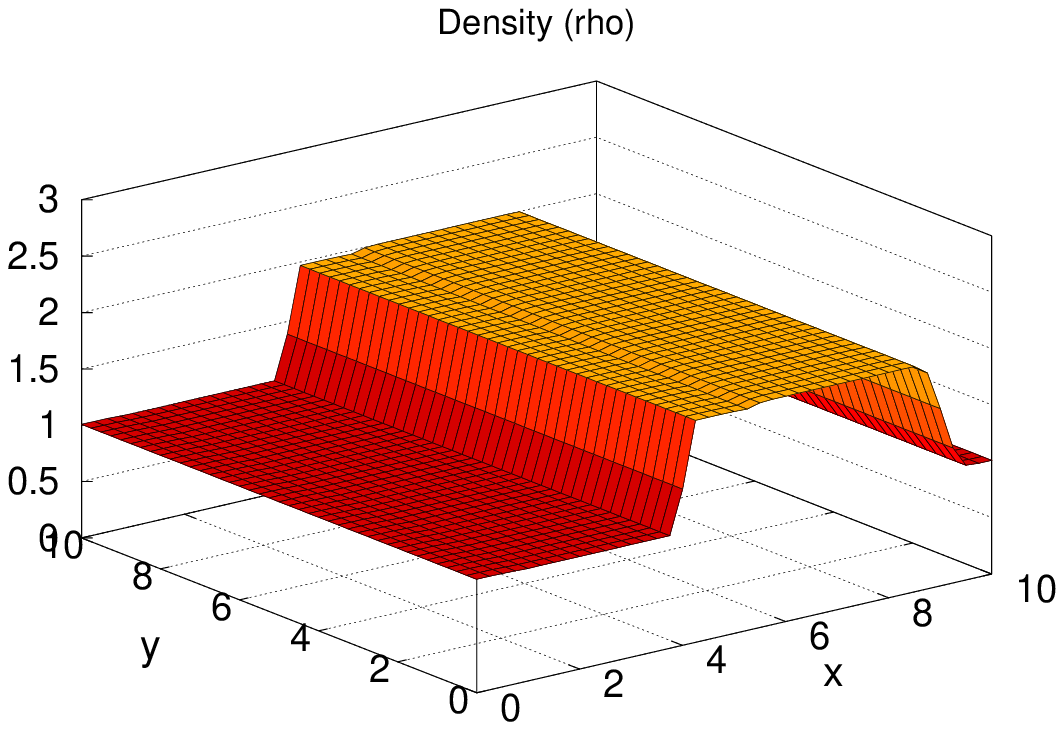}
  \end{minipage}  
  \caption{The particle density in space $\rho$ computed with particle simulations (left)
    and the macroscopic equations (right). We initialize with a Riemann problem
    (\ref{eq:IC_shock}). Numerical parameters for the particle simulations: $N=2.10^6$
    particles, $\Delta t=.01$, $\ep=1/10$, $R=.5$, $Lx=Ly=10$, we take a mean over $10$
    computations. Numerical parameters for the macroscopic model: $\Delta t=.01,\,\Delta
    x=.025$ (CFL=0.416), we use the splitting method. The simulations are run during 2
    time units.}
  \label{fig:exple_2D}
\end{figure}

%%%%%%%%% shock  %%%%%%%%%
\begin{figure}[p]
  %% \centering{{\it Solutions at $t=2$}}\\
  \centering{\includegraphics[scale=.8]{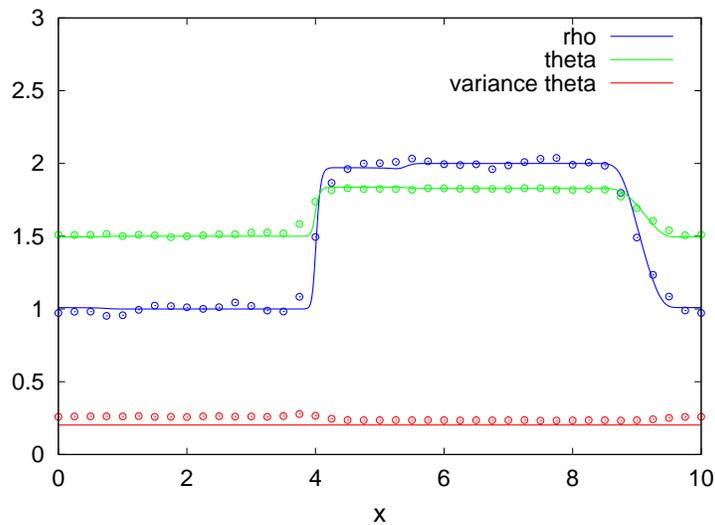}}\\
  \caption{The solution of the Riemann problem (\ref{eq:IC_shock}) with $d=.2$ computed
    with the splitting method (solid line) and with particle simulations (dots). In blue, we
    represent the density $\rho$, in green the flux direction $\theta$ and in red the
    variance of the velocity direction. The parameters are the same as in figure
    \ref{fig:exple_2D}. We only change the representation of the solution
    (1D-representation).}
  \label{fig:choc_d02_particle_split}
\end{figure}

%%%%%%%%%%%%%%%%%%%%%%%%%%%%% 

%%%%%%%%% shock (different solutions) %%%%%%%%%
\begin{figure}[p]
  \includegraphics[scale=1]{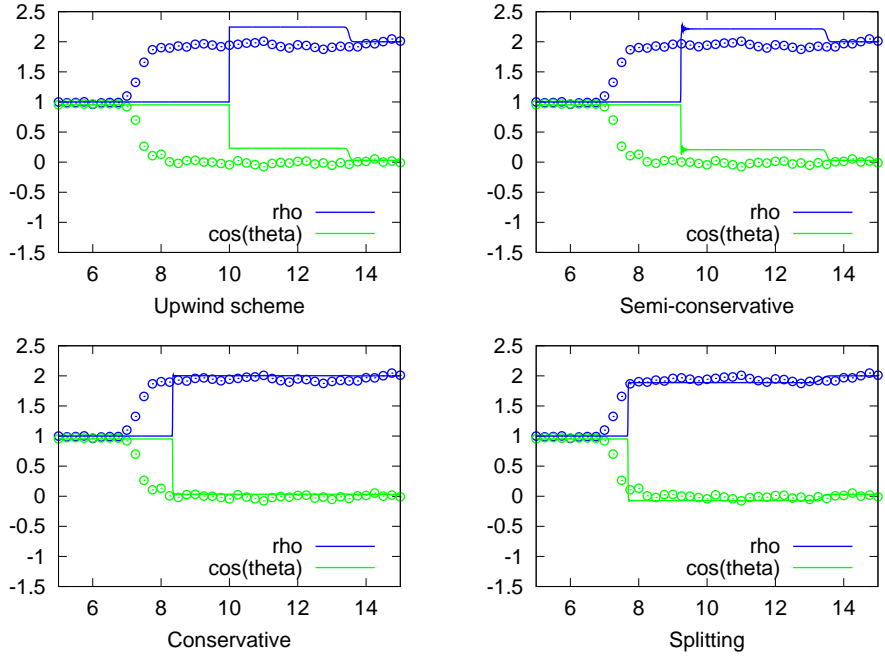}
  \caption{The solutions of the Riemann problem (\ref{eq:riemann_shock2}) with $d=1$ computed
    using the macroscopic model and particle simulations of the microscopic model (see
    figure \ref{fig:choc_d02_particle_split}).  Numerical parameters for the macroscopic
    model: $\Delta t=.01,\,\Delta x=.025$ (CFL=0.778). Numerical parameters for the
    particle simulation: $N=2.10^6$ particles, $\Delta t=.02$, $\ep=.1$, $R=.5$, $Lx=20$
    and $Ly=1$. We take a mean over $50$ simulations. The simulations are run during 6
    time units. Since $d=1$, fluctuations are higher (see figure
    \ref{fig:Mean_speed_diff}), we have to increase the density of particles to reduce
    this effect.}
  \label{fig:choc111_particle}
\end{figure}

%%%%%%%%%%%%%%%%%%%%%%%%%%%%%%%%%%%%%%%%%%%%%% 

%%%%%%%%% theta_pm  %%%%%%%%%
\begin{figure}[p]
  \centering {\it Conservative method} \\
  \centering \includegraphics[scale=.8]{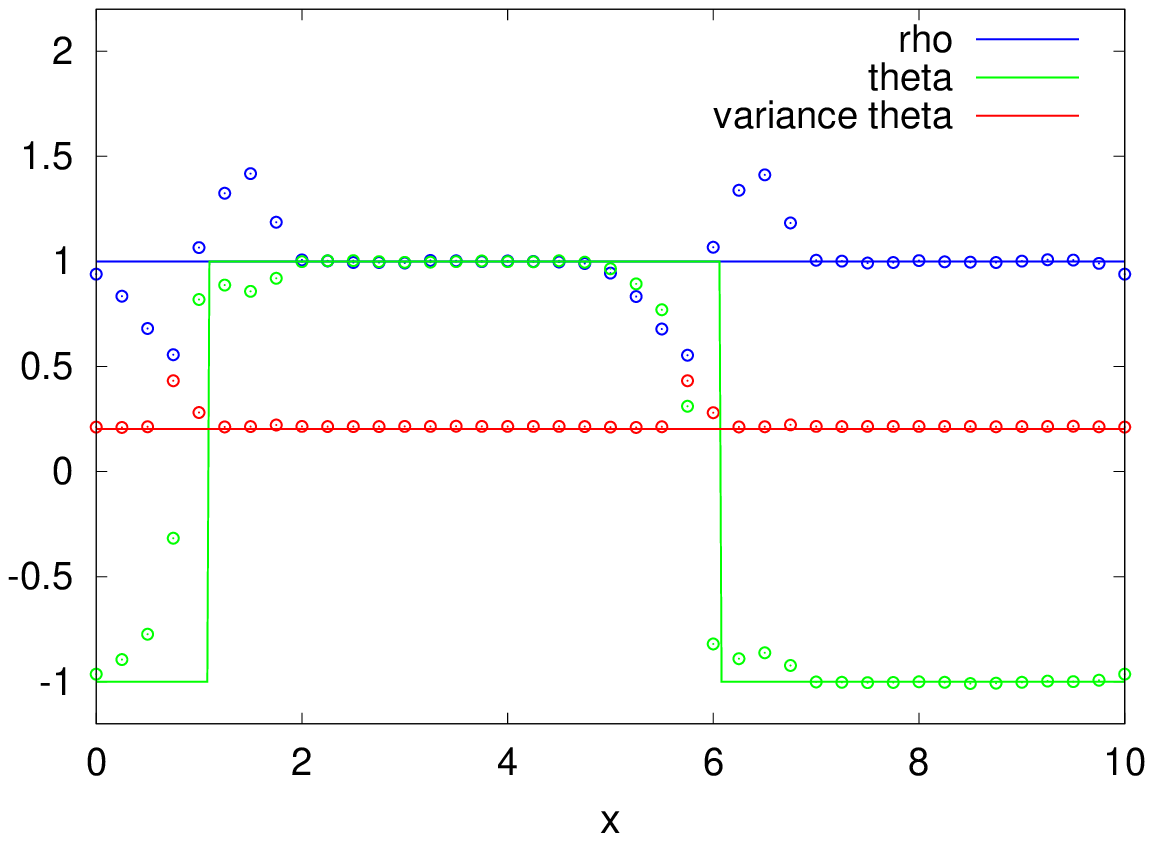}

  \bigskip

  \centering {\it Splitting method}\\
  \centering \includegraphics[scale=.8]{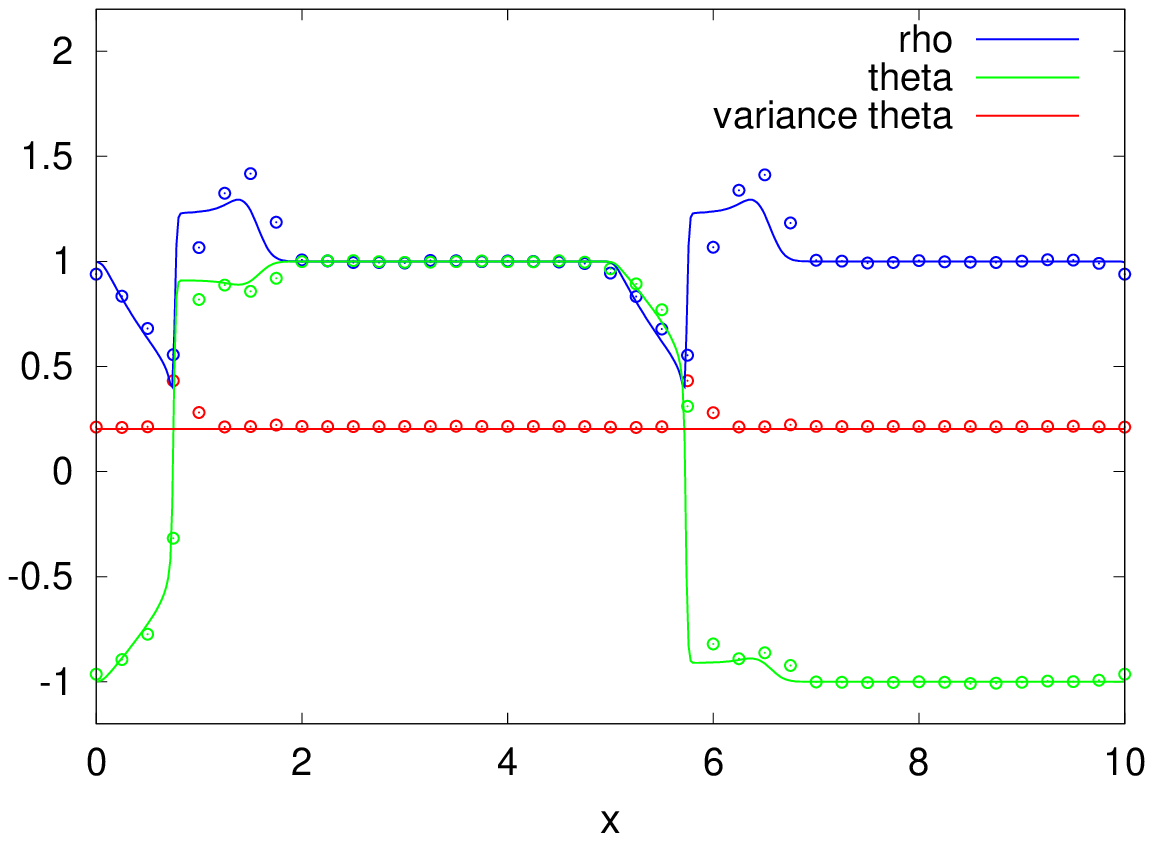}
  
  \caption{The solution of the Riemann problem (\ref{eq:theta_pm}) computed with the
    conservative method (top), the splitting method (down) and with particle simulations
    (dots). Numerical parameters for the macroscopic model: $\Delta t=.01,\,\Delta x=.025$
    (CFL=0.416). Numerical parameters for the particle simulations: $N=10^6$ particles,
    $\Delta t=.01$, $\ep=1/10$, $R=.5$, $Lx=10$, $Ly=1$. We take a mean over $100$
    simulations. The simulations are run during 2 time units.}
  \label{fig:theta_pm}
\end{figure}
%%%%%%%%%%%%%%%%%%%%%%%%%%%%% 

\clearpage

\section{Conclusion}

In this work, we have numerically studied both the microscopic Vicsek model and its
macroscopic version \cite{degond2007ml}. Due to the geometric constraint that the velocity
should be of norm one, the standard theory of hyperbolic systems is not
applicable. Therefore, we have proposed several numerical schemes to solve it. By
comparing the numerical simulations of the microscopic and macroscopic equations, it
appears that the scheme based on a relaxation formulation of the macroscopic model, used in
conjunction with a splitting method is in good agreement with particle simulations. The
other schemes do not show a similar good agreement. In particular, with an initial
condition given by a contact discontinuity, the splitting method and the microscopic model
provide a similar solution which turn to be much more complex than what we could be
expected.

These results confirm the relevance of the macroscopic Vicsek model. Since the CPU time is
much lower with the macroscopic equation, the macroscopic Vicsek model is an effective
tool to simulate the Vicsek dynamics in a dense regime of particles.

Many problems are still open concerning the macroscopic Vicsek model. We have seen that
the splitting method gives results which are in accordance with particle simulations. But,
we have to understand why this particular scheme captures well the particle dynamics
better than the other schemes. Since the macroscopic equation has original
characteristics, this question is challenging. Another point concerns the particle
simulations. We have seen that the particles density has a strong effect on the variance
of the velocity distribution. When the density is low, the variance is larger. The
macroscopic equation does not capture this effect since the variance of the distribution
is constant. Works in progress aims at taking into account this density effect.

\newpage

\appendix

\section{The coefficients $c_1$, $c_2$ and $\lambda$}
\setcounter{equation}{0}
\label{sec:coeff}

The analytical expression of the coefficient $c_1$ involved the distribution of the local
equilibrium $M_{\Omega}$ (\ref{eq:M_von_mises}).  The two other coefficients ($c_2$ and
$\ld$) involve also the solution $g$ of the following elliptic equation \cite{degond2007ml}:
\begin{equation}
  \label{eq:el_fort}
  -(1-x^2)\,\partial_x [ \expo^{x/d} (1-x^2) \partial_x g] + \expo^{x/d} g =
  -(1-x^2)^{3/2} \expo^{x/d},
\end{equation}
on the interval $x \in (-1,1)$. 

If we define the function $h = \frac{g}{\sqrt{1-x^2}}$ and $M(x)=\expo^{\frac{x}{d}}$,
these macroscopic coefficients can be written as:
\begin{eqnarray}
  \label{eq:c_1}
  c_1 &=& < \cos \theta >|_{M} = \frac{ \int_0^{\pi} \cos \theta\,M(\cos \theta) \sin
    \theta\,d\theta} {\int_0^{\pi} M(\cos \theta) \sin \theta\,d\theta},\\
  \label{eq:c_2}
  c_2 &=& < \cos \theta >|_{\sin^2 \theta\, h M} = \frac{ \int_0^{\pi} \cos
    \theta\,\sin^2 \theta\, h(\cos \theta) M(\cos \theta) \sin
    \theta\,d\theta} {\int_0^{\pi} \sin^2 \theta\, h(\cos \theta) M(\cos \theta) \sin
    \theta\,d\theta},\\
  \label{eq:ld}
  %% \ld &=& d <\frac{1}{\nu}>_{\sin^2 \theta\,\nu h M}.
  \ld &=& d.
\end{eqnarray}
In the above expressions, we can see that $0\leq c_1,c_2 \leq 1$.

Now we are going to explore two asymptotics of $g$ when the parameter $d$ is small or large.

\begin{lemma}
  \label{lem:g}
  Let $g$ be the solution of equation (\ref{eq:el_fort}). We have the asymptotics:
  \begin{eqnarray}
    \label{eq:g_d_low}
    g &\stackrel{d\rightarrow 0}{\sim}& d\left[\text{asin}(x) - \frac{\pi}{2}\right] +
    o(d),\\
    \label{eq:g_d_large}
    g &\stackrel{d\rightarrow \infty}{\sim}& -\frac{1}{2}\,\sqrt{1-x^2} +
    \frac{1}{12\,d} x \sqrt{1-x^2} + o\left(\frac{1}{d}\right).
  \end{eqnarray}
\end{lemma}
\begin{proof_formal}
  Introducing the Hilbert space:
  \begin{displaymath}
    V = \{ g \, | \, (1-\mu^2)^{-1/2} g \in L^2(-1,1), \quad (1-\mu^2)^{1/2} \partial_\mu g \in L^2(-1,1) \}
  \end{displaymath}
  we have already seen in \cite{degond2007ml} that there exists a unique solution $g$ of
  (\ref{eq:el_fort}). Moreover this solution is negative.
  
  To derive the asymptotic behavior of $g$ depending on $d$, we first develop (\ref{eq:el_fort}):
  \begin{equation}
    \label{eq:el_fort_dev}
    \partial_x [ (1-x^2) \partial_x g] + (1-x^2)\,\frac{1}{d} \partial_x g -
    \frac{g}{1-x^2} = (1-x^2)^{1/2}. 
  \end{equation}
  When $d\rightarrow 0$, we have:
  \begin{displaymath}
    \partial_x g =0
  \end{displaymath}
  on the interval $[-1+\ep,1-\ep]$ for all $\ep>0$. Since $g$ belongs to $V$, we also have
  the boundary condition $g(-1)=g(1)=0$, so $g$
  converges to $0$ when $d\rightarrow 0$.\\
  To derive the next order of convergence in the limit $d\rightarrow 0$, we normalize $g$
  with $g = d \widetilde{g}$, which gives:
  \begin{eqnarray}
    \label{eq:d_ren}
    d\,\partial_x [ (1-x^2) \partial_x \widetilde{g}] + (1-x^2) \partial_x \widetilde{g} -
    d\,\frac{\widetilde{g}}{1-x^2} = (1-x^2)^{1/2}. 
  \end{eqnarray}
  In the limit $d\rightarrow 0$, we deduce that:
  \begin{eqnarray}
    \label{eq:d_0_ren}
    (1-x^2) \partial_x \widetilde{g} = (1-x^2)^{1/2},
  \end{eqnarray}
  which has an explicit solution: $\widetilde{g} = \text{asin}(x) + c$. Since
  $\widetilde{g}\leq 0$, we have $c\leq -\frac{\pi}{2}$. Numerically, we find 
  that $c=-\frac{\pi}{2}$ but the proof is still open. This formally proves (\ref{eq:g_d_low}).

  \medskip
  
  When $d\rightarrow +\infty$, (\ref{eq:el_fort_dev}) gives:
  \begin{eqnarray}
    \label{eq:d_inf}
    \partial_x [ (1-x^2) \partial_x g_0] - \frac{g_0}{1-x^2} = (1-x^2)^{1/2}. 
  \end{eqnarray}
  A simple calculation shows that $g_0=-\frac{1}{2}\,\sqrt{1-x^2}$ is a solution of
  (\ref{eq:d_inf}).\\
  To derive the next order of convergence, we look at the difference $v = d(g-g_0)$, which
  satisfies (see (\ref{eq:el_fort_dev}) and (\ref{eq:d_inf})):
  \begin{displaymath}
    \partial_x [ (1-x^2) \partial_x v] + (1-x^2)\,\frac{1}{d} \partial_x v -
    \frac{v}{1-x^2} = -(1-x^2)\,\partial_x g_0.
  \end{displaymath}
  In the limit $d \rightarrow +\infty$,  $v$ satisfies:
  \begin{eqnarray}
    \label{eq:g_1_inf}
    \partial_x [ (1-x^2) \partial_x v]  -
    \frac{v}{1-x^2} = -\frac{1}{2} x\sqrt{1-x^2}.
  \end{eqnarray}
  A simple calculation shows that $v = \frac{1}{12} x \sqrt{1-x^2}$ is solution of
  (\ref{eq:g_1_inf}). Therefore we formally have the expression (\ref{eq:g_d_large}) in the proposition.  

\end{proof_formal}

In figure \ref{fig:g}, we compute numerically the function $g$ (\ref{eq:el_fort}). We
use a finite element method with a space step $\Delta x= 10^{-3}$. The two asymptotics of
$g$ when $d\rightarrow 0$ and $d\rightarrow +\infty$ are computed in figure \ref{fig:g1}.

\begin{proposition}
  \label{ppo:c_12}
  The two coefficients $c_1$ and $c_2$ defined (resp.) by the equations (\ref{eq:c_1}) and
  (\ref{eq:c_2}) satisfy the following asymptotics:
  \begin{eqnarray}
    \label{eq:c_1_asymp_0}
    c_1 &\stackrel{d \rightarrow 0}{\sim}& 1-d + O(d^2), \\
    \label{eq:c_1_asymp_inf}
    c_1 &\stackrel{d \rightarrow +\infty}{\sim}& \frac{1}{3d} + O\left(\frac{1}{d^2}\right),\\
    \label{eq:c_2_asymp_inf}
    c_2 &\stackrel{d \rightarrow \infty}{\sim}& \frac{1}{6d} + o\left(\frac{1}{d}\right).
  \end{eqnarray}
\end{proposition}
\begin{proof}
  We have an explicit expression for the coefficient $c_1$ using the change of unknowns
  $x=\cos(\theta)$:
  \begin{eqnarray}
    c_1 &=& \text{coth}\left(\frac{1}{d} \right)-d,
  \end{eqnarray}
  where $\text{coth}(s) = \frac{\expo^{s}+\expo^{-s}}{\expo^{s}-\expo^{-s}}$. The
  expressions of (\ref{eq:c_1_asymp_0}) and (\ref{eq:c_1_asymp_inf}) are simply deduced by
  a Taylor expansion of the last expression.

  For the coefficient $c_2$, we insert the development of $g$ (\ref{eq:g_d_large}) in
  expression (\ref{eq:c_2}).
\end{proof}
\begin{remark}
  The behavior of $c_2$ when $d\rightarrow 0$ is more difficult to analyze. The density
  probability $\sin \theta\,h\,M$ used in formula (\ref{eq:c_2}) becomes singular in this
  limit. Nevertheless, due to the expression of $M_{\Omega}$, the density converges to a
  Dirac delta at $0$ which explains why $c_2 \stackrel{d \rightarrow 0}{\sim} 1$. To
  capture the next order of convergence, we need to find the second order correction of
  $g$ in the limit $d\rightarrow 0$ which is not available. However, numerically we find
  that: % (kop ref amic)
  \begin{displaymath}
    c_2 \stackrel{d \rightarrow 0}{\sim} 1-2d +o(d).
  \end{displaymath}
\end{remark}
In figure \ref{fig:div_c1}, we numerically compute the coefficients $c_2/c_1$, $\ld/c_1$
and their asymptotics.

\begin{figure}[hp]
  \centering
  \includegraphics[scale=.8]{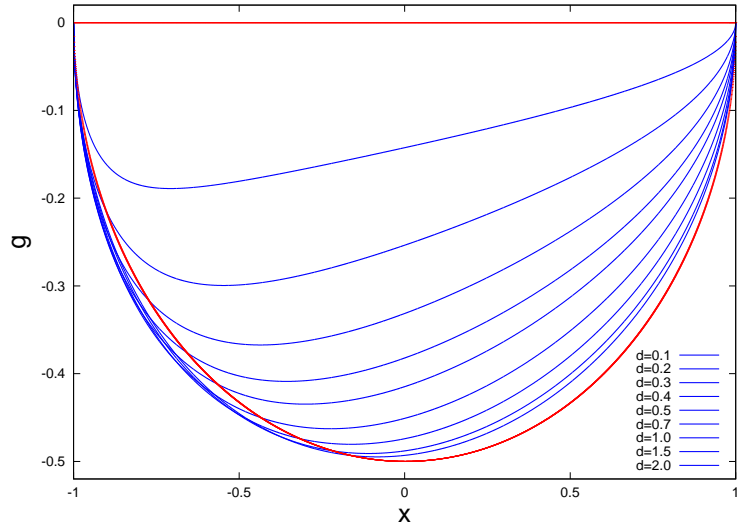}
  \caption{The numerical solution $g$ of (\ref{eq:el_fort}) for different values of the
    parameter $d$. We have the following asymptotics (see lemma \ref{lem:g}):
    $g\stackrel{d \rightarrow 0}{\longrightarrow} 0$ and $g\stackrel{d \rightarrow
      \infty}{\longrightarrow} -\frac{1}{2}\sqrt{1-x^2}$.}
  \label{fig:g}
\end{figure}

\begin{figure}[p]
  % \centering
  \includegraphics[scale=.56]{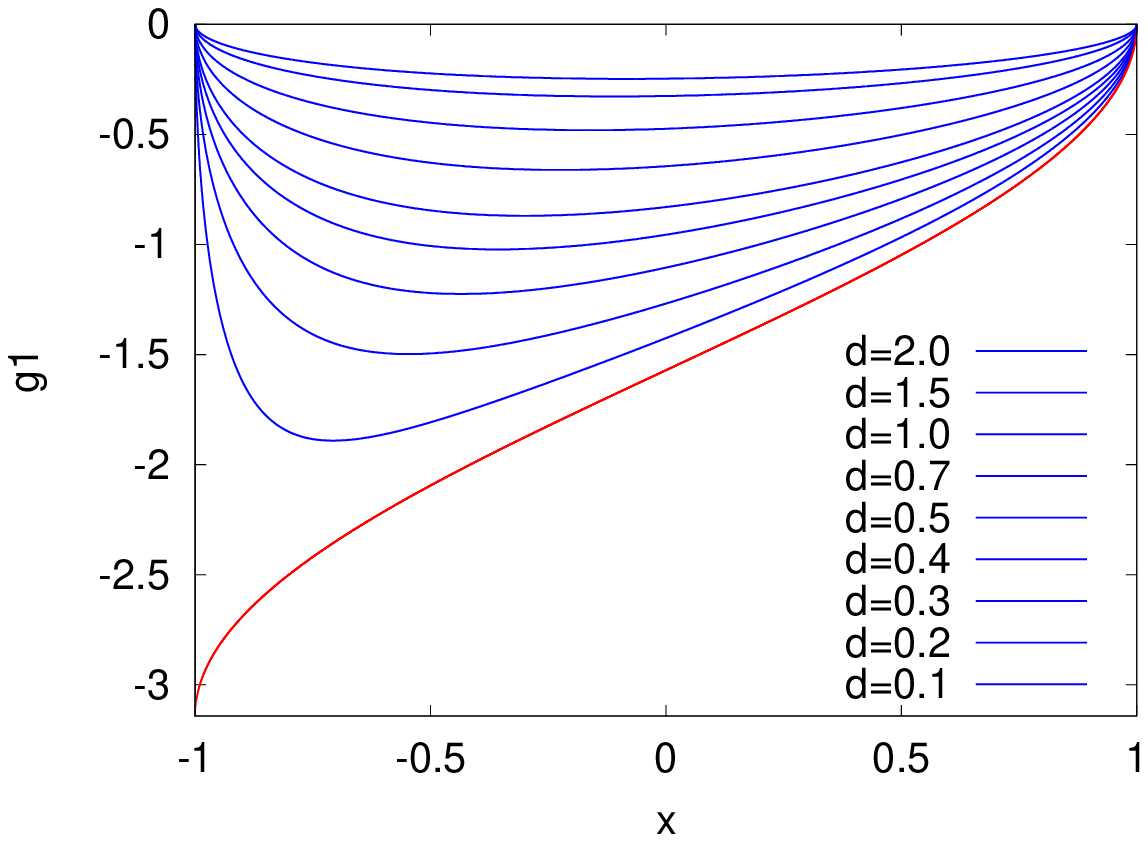}
  \includegraphics[scale=.56]{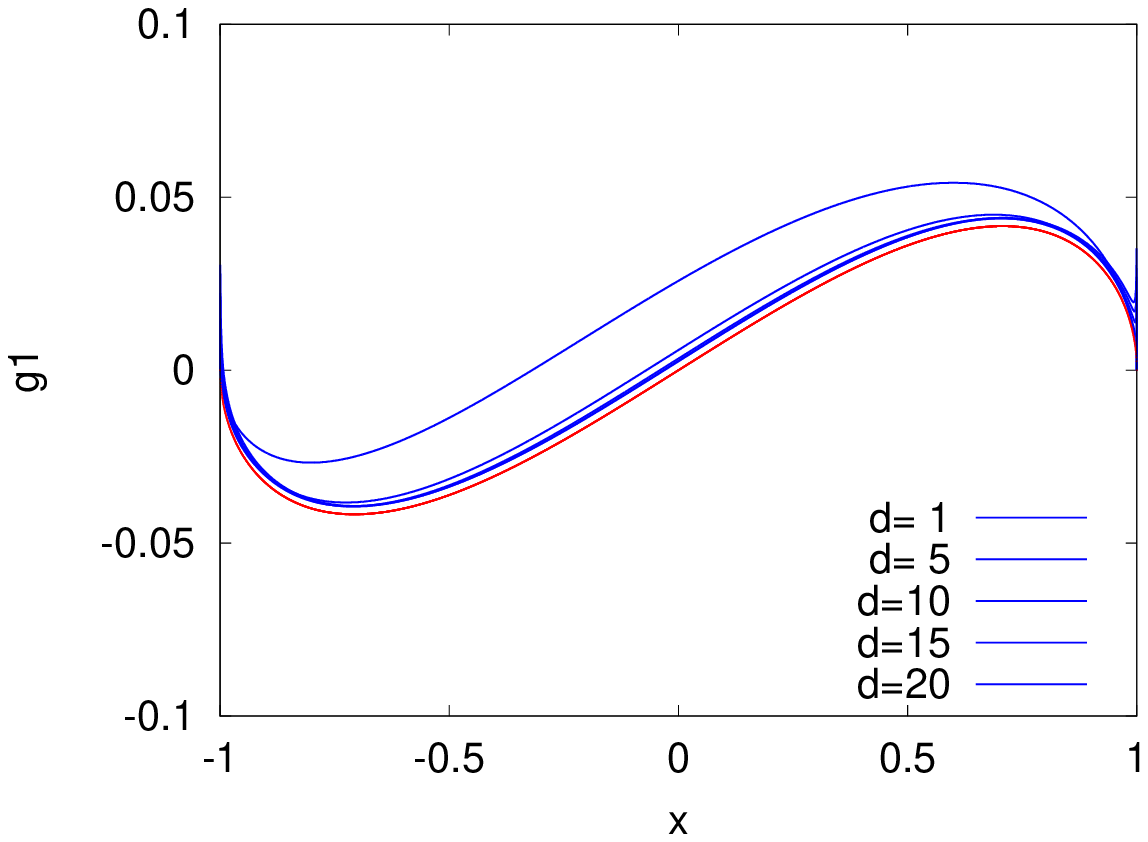}
  \caption{Left figure: the first correction $g_1$ of $g$ when $d\rightarrow 0$. The red
    curve is the theoretical asymptotic limit: $g_1 = g/d \stackrel{d\rightarrow
      0}{\longrightarrow} \text{asin}(x) - \frac{\pi}{2}$ (see lemma \ref{lem:g}).
    Right figure: the first correction $g_1$ of $g$ when $d \rightarrow +\infty$. The red
    curve is the theoretical asymptotic: $g_1 = d(g + \frac{1}{2}\,\sqrt{1-x^2})
    \stackrel{d\rightarrow 0}{\longrightarrow} \frac{1}{12} x \sqrt{1-x^2}$ (see lemma
    \ref{lem:g}).}
  \label{fig:g1}
\end{figure}

\begin{figure}
  \includegraphics[scale=.37]{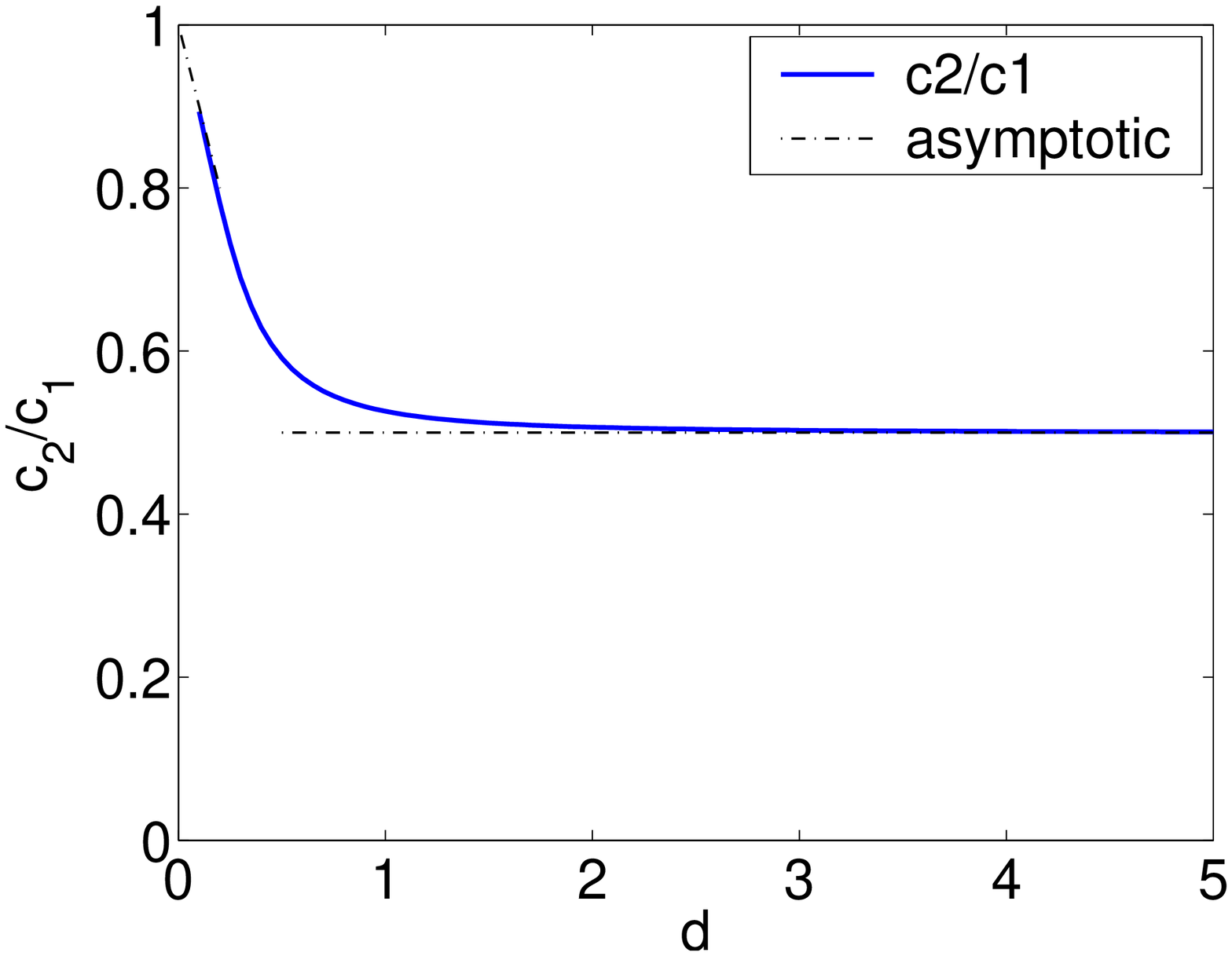}
  \includegraphics[scale=.37]{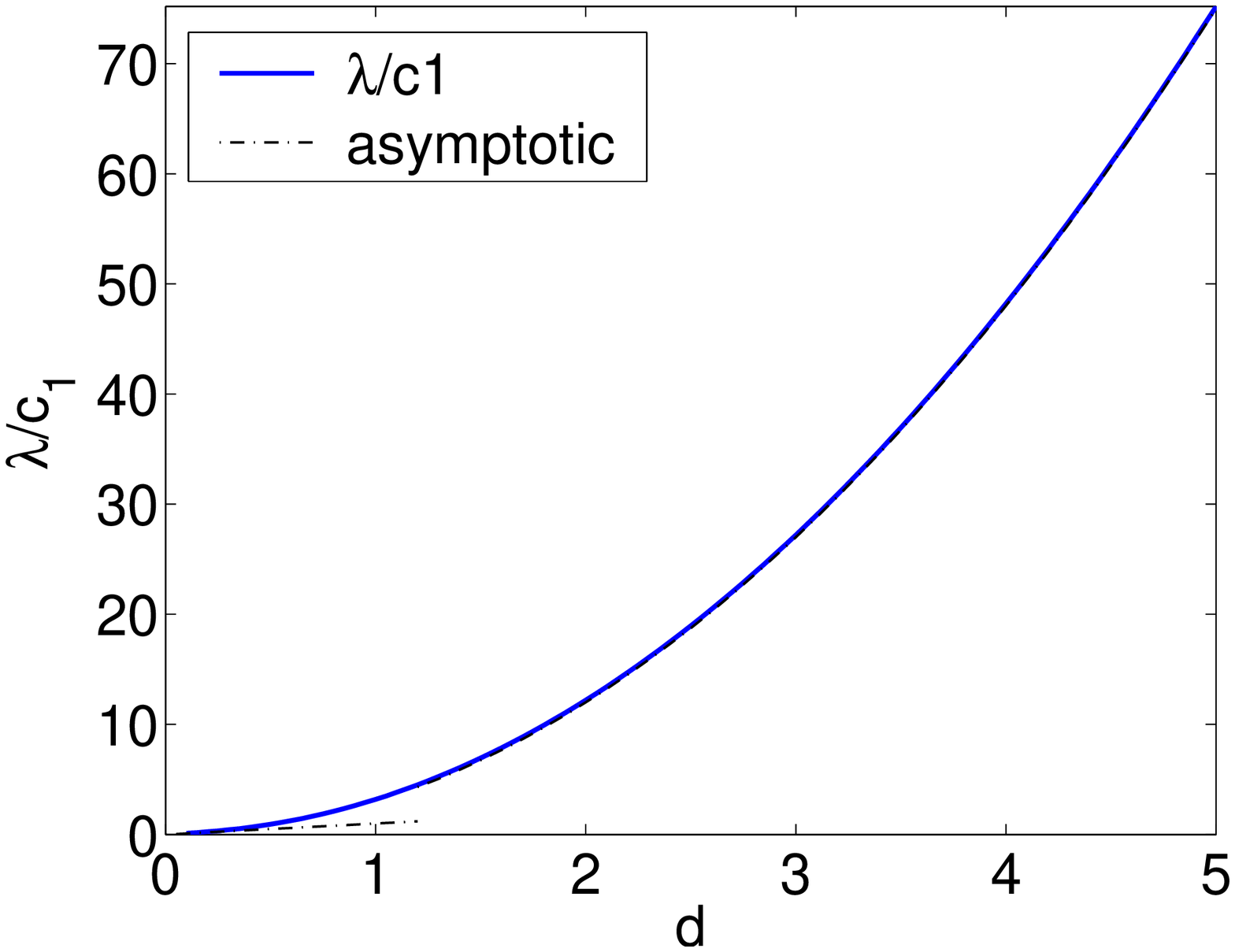}
  \caption{The ratio $c_2/c_1$ and $\ld/c_1$ (solid lines) and their two asymptotics (dashed lines) (see
    proposition \ref{ppo:c_12}) computed with $\Delta x=10^{-3}$.}
  \label{fig:div_c1}
\end{figure}

\clearpage

\section{Special solution of the MV model}
\label{sec:moulin}
\setcounter{equation}{0}

In this appendix, a vortex configuration is exhibited as a stationary solution of the MV
model (\ref{eq:macro_rho})-(\ref{eq:macro_omega})-(\ref{eq:macro_constraint}) in dimension
2. A stationary state of the MV model has to satisfy:
\begin{equation}
  \label{eq:stationnary}
  \begin{array}{ll}
    & \nabla_{x} \cdot (\rho \Omega)  = 0, \\
    & c(\Omega \cdot \nabla_{x}) \Omega +  \lambda  \,  (\mbox{Id} - \Omega \otimes
    \Omega) \frac{\nabla_{x} \rho}{\rho} = 0.
  \end{array}
\end{equation}
Introducing the polar coordinates, $\rho(r,\theta)$, $\Omega(r,\theta) = f_r(r,\theta) \vec{e}_r +
f_\theta(r,\theta)\vec{e}_\theta$, where $\vec{e}_r =
\left(\cos\theta,\sin\theta\right)^{T}$ and $\vec{e}_\theta = \left(- \sin \theta,\cos
  \theta\right)^{T}$, we are able to formulate the proposition:

\begin{proposition}
  \label{ppo:vortex}
  The following initial condition is a stationary state of the MV model (\ref{eq:stationnary}):
  \begin{equation}
    \label{eq:init_vortex}
    \rho(r) = C\, r^{c/\ld} \qquad, \qquad \Omega = \vec{e}_{\theta},  % = C (\sqrt{x2+y2})^{c/\ld}
  \end{equation}
  where $C$ is a constant.
\end{proposition}
\begin{proof}
  With the expression of $\rho$ and $\Omega$ given by (\ref{eq:init_vortex}), the
  divergence of the mass is zero and the gradient of $\rho$ is orthogonal to $\Omega$,
  therefore the system (\ref{eq:stationnary}) reduces to:
  \begin{equation}
    \label{eq:reduc}
    c(\Omega \cdot \nabla_{x}) \Omega   +  \lambda \frac{\nabla_{x} \rho}{\rho} = 0,
  \end{equation}
  or in polar coordinates:
  \begin{displaymath}
    c \frac{1}{r} \partial_\theta \,\vec{e}_\theta \, + \, \ld
    \frac{\rho'(r)}{\rho(r)} \vec{e}_r = 0.
  \end{displaymath}
  Since $\partial_\theta \vec{e}_\theta = -\vec{e}_r$, we can easily check that the solution
  of this equation is given by $\rho(r) = C\, r^{c/\ld}$.
\end{proof}

\section{Numerical schemes for particle simulations}
\label{sec:num_scheme_ps}
\setcounter{equation}{0}

In the limit $\ep\rightarrow 0$, an explicit Euler method for the differential system
(\ref{Cont_pos_bis})-(\ref{Cont_orient_bis}) imposes a restriction time step condition of
$\frac{1}{\ep} \Delta t<1$. Therefore, we develop an implicit scheme for this system.  The
idea is to go back to the original Vicsek model (see \cite{degond2007ml}). We use the
formulation:
\begin{equation}
  \label{eq:mid_point}
  \frac{\omega^{n+1}-\omega^n}{\Delta t} = (\mbox{Id} - \omega^{n+1/2} \otimes
  \omega^{n+1/2})(\bar \omega^n - \omega^n)
\end{equation}
where $\omega^{n+1/2} = \frac{\omega^{n}+\omega^{n+1}} {|\omega^{n}+\omega^{n+1}|}$ and
$\bar \omega^n$ is the average velocity (\ref{CV_moy_3_bis}). When $\Delta t=1$, we recover
exactly the original Vicsek model \cite{vicsek1995ntp}. (\ref{eq:mid_point}) can in fact
be solved explicitly. First, we have to remember that $\omega^{n+1}$ belongs to the unit circle
(i.e. $|\omega^{n+1}|=1$). Then we use that $\omega^{n+1}-\omega^n$ is the orthogonal
projection of $(\bar \omega^n - \omega^n)\Delta t$ on the orthogonal plan of
$\omega^{n+1/2}$. Therefore $\omega^{n+1}$ and $\omega_{n}$ are on the circle
$\mathcal{C}$ with center $B = \omega_{n} + \frac{(\bar \omega^n - \omega^n)\Delta t}{2}$ and radius
$\left|\frac{(\bar{\omega}^n - \omega^n)\Delta t}{2}\right|$ (see figure \ref{fig:tech_circ}). This
fully defines $\omega^{n+1}$ since $\omega^{n}$ and $\omega^{n+1}$ are the two
intersection points of the unit circle and the circle $\mathcal{C}$. Denoting $\theta$ the
angle of the unit vector $\omega$, we easily check that we have in terms of angles:
\begin{displaymath}
  \theta^{n+1} = \theta^n + 2 \widehat{(\omega^n,B)}.
\end{displaymath}
To take into account the effect of the noise, we simply add a random variable:
\begin{equation}
  \label{eq:discrete_micro_angle}
  \theta^{n+1} = \theta^n + 2 \widehat{(\omega^n,B)} + \sqrt{2d\,\Delta t}\, \epsilon_n
\end{equation}
where $\epsilon_n$ is a standard normal distribution independent of $\theta^n$.

\begin{figure}[ht]
  \centering
  \includegraphics[scale=.6]{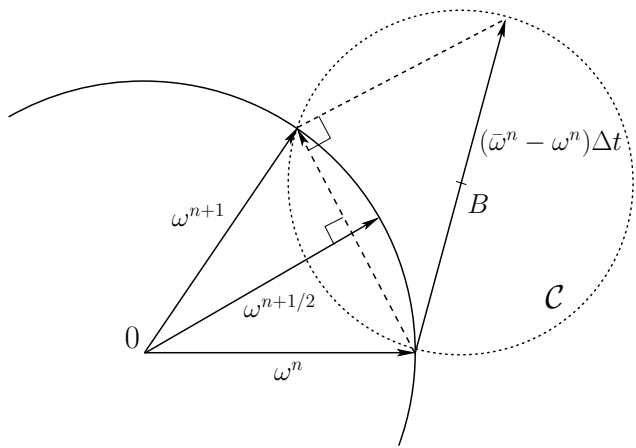}
  \caption{Illustration of the geometric method to solve explicitly equation
    (\ref{eq:mid_point}).}
  \label{fig:tech_circ}
\end{figure}

\medskip

\noindent {\bf Algorithm used to solve a Riemann problem with particles.}

\begin{enumerate}
\item Choose a Riemann problem $(\rho_l,\,\theta_l)$ and $(\rho_r,\,\theta_r)$.
\item Initiate $N$ particles $(x_k,\omega_k)_{k=1..N}$ according to the distributions
  $\rho_l M_{\Omega_l}$ and $\rho_r M_{\Omega_r}$.
\item Let evolve the particles in time using the time-discretization
  (\ref{eq:discrete_micro_angle}) of equation (\ref{Cont_orient_bis}).
\item Compute the mass $\rho$ and the direction of the flux $\Omega$ using
  Particle-In-Cell method \cite{hockney1988csu} in order to compare the simulation with
  the one of the MV model.
\end{enumerate}

\end{document}